\documentclass[11pt]{amsart}
\usepackage[T1]{fontenc}
\usepackage{bbm}
\usepackage{amsmath, amssymb, amsfonts, amsthm, verbatim, dsfont, graphicx}
\usepackage[hidelinks]{hyperref}
\usepackage{color}
\usepackage{mathtools}
\mathtoolsset{showonlyrefs}
\usepackage{array}
\usepackage{tabu}
\numberwithin{equation}{section}

\makeatletter
\renewcommand\tableofcontents{%
    \section*{\huge{Table of Contents}
        \@mkboth{%
           \MakeUppercase\contentsname}{\MakeUppercase\contentsname}}
    \@starttoc{toc}%
    } 
\makeatother

\newsavebox\MBox

\newtheorem{theorem}{Theorem}[section]
\newtheorem{proposition}[theorem]{Proposition}
\newtheorem{corollary}[theorem]{Corollary}
\newtheorem{lemma}[theorem]{Lemma}

\theoremstyle{definition} \newtheorem{remark}[theorem]{Remark}

\numberwithin{figure}{section}

\newcommand{\bR}{{\mathbb R}}
\newcommand{\bN}{{\mathbb N}}
\newcommand{\bC}{{\mathbb C}}

\newcommand{\bZ}{{\mathbb Z}}

\newcommand{\bP}{{\mathbb P}}
\newcommand{\bS}{{\mathbb S}}

\newcommand{\cF}{{\mathcal{F}}}

\newcommand{\cI}{{\mathcal{I}}}

\title[a.s. scattering for the 4D energy-critical NLW with radial data]{Almost sure scattering for the 4D energy-critical defocusing nonlinear wave equation with radial data}

\author[B. Dodson]{Benjamin Dodson}
\address{Department of Mathematics \\ Johns Hopkins University \\ Krieger Hall 214 \\ 3400 N. Charles Street \\ Baltimore, MD 21218, USA}
\email{bdodson4@jhu.edu}

\author[J. L\"uhrmann]{Jonas L\"uhrmann}
\address{Department of Mathematics \\ Johns Hopkins University \\ Krieger Hall 219 \\ 3400 N. Charles Street \\ Baltimore, MD 21218, USA}
\email{luehrmann@math.jhu.edu}

\author[D. Mendelson]{Dana Mendelson}
\address{Department of Mathematics \\ University of Chicago \\ Eckhart 220 \\ 5734 S. University Ave \\ Chicago, IL 60637, USA}
\email{dana@math.uchicago.edu}

\thanks{The first author was supported in part by National Science Foundation grant DMS-1500424.}

\begin{document}

\begin{abstract}
 We consider the energy-critical defocusing nonlinear wave equation on $\bR^4$ and establish almost sure global existence and scattering for randomized radially symmetric initial data in $H^s_x(\bR^4) \times H^{s-1}_x(\bR^4)$ for $\frac{1}{2} < s < 1$. This is the first almost sure scattering result for an energy-critical dispersive or hyperbolic equation with scaling super-critical initial data. The proof is based on the introduction of an approximate Morawetz estimate to the random data setting and new large deviation estimates for the free wave evolution of randomized radially symmetric data.
\end{abstract}

\maketitle

\section{Introduction} \label{sec:introduction}
\setcounter{equation}{0}

We consider the Cauchy problem for the energy-critical defocusing nonlinear wave equation in four space dimensions
\begin{equation} \label{equ:ivp}
 \left\{ \begin{aligned}
  -\partial_t^2 u + \Delta u &= u^3 \text{ on } \bR \times \bR^4, \\
  (u, \partial_t u)|_{t=0} &= (f_0, f_1) \in H^s_x(\bR^4) \times H^{s-1}_x(\bR^4).
 \end{aligned} \right.
\end{equation}
The purpose of this article is to study the asymptotic behavior of solutions to \eqref{equ:ivp} for rough and random initial data below the scaling critical regularity. Our main result establishes the almost sure global existence and scattering of solutions to~\eqref{equ:ivp} with respect to a unit-scale randomization in frequency space of radially symmetric initial data in $H^s_x(\bR^4) \times H^{s-1}_x(\bR^4)$ for $\frac{1}{2} < s < 1$.

\medskip

The equation \eqref{equ:ivp} is invariant under the scaling
\begin{equation} \label{equ:scaling_cubic}
 u(t,x) \mapsto \lambda u(\lambda t, \lambda x) \quad \text{for } \lambda > 0,
\end{equation}
and the scaling critical regularity $s_c = 1$ is by definition such that the corresponding homogeneous Sobolev norms of the initial data are left invariant by the scaling transformation~\eqref{equ:scaling_cubic}. Sufficiently regular solutions to~\eqref{equ:ivp} conserve the energy 
\begin{equation} \label{equ:energy_functional}
 E(u) =  \int_{\bR^4} \frac{1}{2} |\nabla_x u|^2 + \frac{1}{2} |\partial_t u|^2 + \frac{1}{4} |u|^{4} \, dx,
\end{equation}
and since this energy functional is also invariant under the scaling~\eqref{equ:scaling_cubic}, the Cauchy problem for~\eqref{equ:ivp} is referred to as energy-critical.

\medskip 

The energy-critical defocusing nonlinear wave equation has by now become the subject of a vast body of literature that we cannot review here in its entirety. We therefore focus on the most relevant results for this paper. In \cite{Lindblad_Sogge}, Lindblad and Sogge construct local strong solutions to \eqref{equ:ivp} for sub-critical and critical regularities $s \geq s_c$ using Strichartz estimates for the wave equation. When $s < s_c$, this is the super-critical regime and the local well-posedness arguments based on Strichartz estimates break down. In fact, the Cauchy problem~\eqref{equ:ivp} is ill-behaved below the scaling critical regularity, see for instance \cite{CCT}, \cite{L} and \cite{IMM}. For initial data in the energy class, i.e. when $s=1$, global well-posedness, scattering, a priori bounds on scattering norms and concentration compactness properties of the solutions to~\eqref{equ:ivp} have been established in the classical works of Struwe~\cite{Struwe}, Grillakis~\cite{Grillakis}, Ginibre-Soffer-Velo~\cite{Ginibre_Soffer_Velo}, Shatah-Struwe~\cite{Shatah_Struwe}, Bahouri-Shatah~\cite{Bahouri_Shatah}, Bahouri-G\'erard~\cite{Bahouri_Gerard}, Nakanishi~\cite{Nakanishi}, and Tao~\cite{Tao}.

\medskip

Even though the nonlinear wave equation \eqref{equ:ivp} is ill-posed below the scaling critical regularity $s_c = 1$, it is sometimes possible to construct ``unique'' local and even global solutions for suitably randomized initial data, and thereby conclude that large sets of initial data of super-critical regularity do indeed lead to global solutions. We refer to Remark \ref{rem:uniqueness} for the precise notion of uniqueness in this context. This approach was initiated by Bourgain~\cite{B94, B96} for the periodic nonlinear Schr\"odinger equation in dimensions one and two, building upon the constructions of invariant measures in~\cite{Glimm_Jaffe} and \cite{LRS}. Subsequently, Burq-Tzvetkov studied the cubic nonlinear wave equation on a three-dimensional compact manifold. By randomizing with respect to an orthonormal eigenbasis of the Laplacian, they proved almost sure local well-posedness for initial data of super-critical regularity~\cite{BT1}. Using invariant measure considerations they also established almost sure global well-posedness for radial random data on the three-dimensional Euclidean unit ball~\cite{BT2}. In recent years extensive work has been done on Cauchy problems for dispersive and hyperbolic equations with random initial data, both in compact and non-compact settings. 

\medskip

On Euclidean space, which is our current setting, no orthonormal basis of eigenfunctions of the Laplacian exists, and there is no known invariant measure for \eqref{equ:ivp}. Many previous results on Euclidean space involve first considering a related equation in a setting where an orthonormal eigenbasis for the Laplacian exists, see for instance \cite{BTT, D1, Poiret, Poiret_Robert_Thomann, Suzzoni1, Suzzoni2, Thomann}. It is also possible, however, to randomize initial data directly on Euclidean space using a unit-scale decomposition of frequency space, see definition~\eqref{equ:bighsrandomization} below. As in the works of Burq-Tzvetkov discussed above, the main feature of this unit-scale randomization is that the free wave evolution of the randomized data possesses almost surely much better integrability properties. In many cases this approach has yielded probabilistic local and global existence results for super-critical initial data, see for example \cite{ZF, LM, BOP2, BOP1, Pocovnicu, OP, LM2, Brereton}, as well as Remark \ref{rmk:unit-scale} for further discussion. We take this direct approach in our present work. 

\medskip

We mainly restrict the following overview to probabilistic results for power-type nonlinear wave equations on Euclidean space. In \cite{LM}, the second and third authors proved almost sure global existence for energy sub-critical nonlinear wave equations on~$\bR^3$ with scaling super-critical initial data using a probabilistic local existence argument and Bourgain's high-low frequency decomposition~\cite{B98}. This approach was previously introduced by Colliander-Oh~\cite{CO} in the context of the one-dimensional periodic defocusing cubic nonlinear Schr\"odinger equation. The results of \cite{LM} were later improved by the second and third authors~\cite{LM2} and by Sun-Xia~\cite{SX}. For energy-critical nonlinearities, Pocovnicu~\cite{Pocovnicu} established almost sure global well-posedness for the energy-critical defocusing nonlinear wave equation on $\bR^4$ for regularities $s > 0$ and on $\bR^5$ for $s \geq 0$. Subsequently, Oh-Pocovnicu~\cite{OP} treated the energy-critical nonlinear wave equation on $\bR^3$ for regularities $s > \frac{1}{2}$. These proofs use probabilistic perturbation theory together with probabilistic a priori energy bounds. Finally, we refer to \cite{BOP2, BOP1, Brereton} for probabilistic well-posedness results for nonlinear Schr\"odinger equations on Euclidean space. 

\medskip 

Of the aforementioned results, those relying on the unit-scale randomization procedure do not yield any almost sure statement about the asymptotic behavior of the solutions for random data of super-critical regularity. In many cases probabilistic small data scattering is known. For energy-critical defocusing nonlinear wave equations, which is our current context, this was first proved by the second and third authors~\cite[Theorem 1.6]{LM} on $\bR^3$, with an identical proof working also on $\bR^d$, $d = 4,5$, see~\cite[Theorem 1.2]{Pocovnicu}. In Theorem~\ref{thm:random_scattering} we establish the first almost sure scattering result for an energy-critical equation, more precisely, we prove almost sure scattering for the defocusing nonlinear wave equation on $\bR^4$ for randomized radially symmetric data of super-critical regularity. The main novelties of our work are the introduction of an approximate Morawetz identity to the random data setting, as well as the introduction of new large deviation estimates for global weighted space-time norms of the free wave evolution of randomized radially symmetric data.

\subsection{Randomization procedure}

Before turning to the statements of our main results, we introduce the randomization procedure for the initial data. Let $\psi \in C_c^{\infty}(\bR^4)$ be an even, non-negative function with $\text{supp} (\psi) \subseteq B(0,1)$ and such that 
\[
 \sum_{k \in \bZ^4} \psi(\xi - k) = 1 \quad \text{for all } \xi \in \bR^4.
\]
Let $s \in \bR$ and let $f \in H^s_x(\bR^4)$. For every $k \in \bZ^4$, we define the function $P_k f: \bR^4 \rightarrow \bC$ by
\begin{equation}\label{equ:unit_scale_proj}
 (P_k f)(x) = \cF^{-1} \left( \psi(\xi - k) \hat{f}(\xi) \right)(x) \quad \text{for } x \in \bR^4.
\end{equation} 
By requiring the cut-off $\psi$ to be even, we ensure that real-valued functions $f$ satisfy the symmetry condition
\begin{equation} \label{equ:symmetry}
 \overline{P_k f} = P_{-k} f.
\end{equation}
We crucially exploit that these Fourier projections satisfy a unit-scale Bernstein inequality, namely for all $1 \leq r_1 \leq r_2 \leq \infty$ we have that 
\begin{align} \label{equ:unit_scale_bernstein}
  \|P_k f\|_{L^{r_2}_x(\bR^4)} \leq C(r_1, r_2) \|P_k f\|_{L^{r_1}_x(\bR^4)}
\end{align}
with a constant which is independent of $k \in \bZ^4$.

\medskip

We let $\{ (g_k, h_k) \}_{k \in \bZ^4}$  be a sequence of zero-mean, complex-valued Gaussian random variables on a probability space $(\Omega, {\mathcal A}, \bP)$ with the symmetry condition $g_{-k} = \overline {g_k}$ and $h_{-k} = \overline {h_k}$  for all $k \in \bZ^4$. We assume that $\{g_0, \textup{Re}(g_k), \textup{Im}(g_k)\}_{k \in \cI}$ are independent, zero-mean, real-valued random variables, where $\cI \subset \bZ^4$ is such that we have a \textit{disjoint} union $\bZ^4 = \cI \cup (-\cI) \cup \{0\}$, and similarly for the $h_k$. 

\medskip

Given a pair of real-valued functions $(f_0, f_1) \in H^s_x(\bR^4) \times H^{s-1}_x(\bR^4)$ for $s \in \bR$, we define its randomization by
\begin{equation} \label{equ:bighsrandomization} 
 (f_0^{\omega}, f_1^{\omega}) := \biggl( \sum_{k \in \bZ^4} g_k(\omega) P_k f_0, \sum_{k \in \bZ^4} h_k(\omega) P_k f_1 \biggr).
\end{equation}
This quantity is understood as a Cauchy limit in $L^2_\omega\bigl(\Omega; H^s_x(\bR^4) \times H^{s-1}_x(\bR^4)\bigr)$, and  in the sequel, we will restrict ourselves to a subset $\Sigma \subset \Omega$ with $\bP(\Sigma) = 1$ such that $(f_0^\omega, f_1^\omega) \in H^s_x(\bR^4) \times H^{s-1}_x(\bR^4)$ for every $\omega \in \Sigma$. The randomization almost surely does not regularize at the level of Sobolev spaces, see for instance \cite[Lemma B.1]{BT1}, and the symmetry assumptions on the random variables as well as \eqref{equ:symmetry} ensure that the randomization of real-valued initial data is real-valued. 

\medskip 

We denote the free wave evolution of the random initial data $(f^\omega_0, f^\omega_1)$ by
\begin{align} \label{equ:free_evolution}
 S(t)(f^\omega_0, f^\omega_1) = \cos(t|\nabla|) f_0^\omega + \frac{\sin(t|\nabla|)}{|\nabla|} f_1^\omega.
\end{align}
As mentioned above, an important feature of the unit-scale randomization~\eqref{equ:bighsrandomization} is that the free wave evolution of the random data satisfies significantly improved space-time integrability properties, as demonstrated by Proposition~\ref{prop:large_deviation_L3L6} and Proposition~\ref{prop:large_deviation_L2Linfty} in this work. This property is akin to the classical results of Paley and Zygmund~\cite{Paley_Zygmund1} on the improved integrability of random Fourier series.

\begin{remark} \label{rmk:unit-scale}
We use the terminology ``unit-scale randomization'' for \eqref{equ:bighsrandomization} in reference to the unit-sized supports of the frequency multipliers $\psi_k(\xi)$. Similar randomizations have previously been used in Euclidean space, first in \cite{ZF}, and subsequently in \cite{LM}, \cite{BOP2, BOP1}, and several works referenced above. In \cite{BOP2, BOP1} the terminology ``Wiener randomization'' was used to highlight the connection between the randomization~\eqref{equ:bighsrandomization}, the Wiener decomposition~\cite{Wiener}, and the modulation spaces introduced by H. Feichtinger~\cite{Feichtinger}.
\end{remark}

\begin{remark}
One may also randomize with respect to a more general sequence of random variables $\{ (g_k, h_k) \}_{k\in\bZ^4}$ satisfying the following condition: let $\mu_{k}$ and $\nu_{k}$ be the joint distributions of the real and imaginary parts of the random variables $g_k$ and $h_k$, respectively, so that there exists $c > 0$ with
\begin{equation} \label{eq:rvassumption}
 \left| \int_{-\infty}^{+\infty} e^{\gamma x} \, d\mu_{k}(x) \right| \leq e^{c \gamma^2} \quad \text{for all } \gamma \in \bR \text{ and for all } k \in \bZ^4,
\end{equation}
and similarly for $\nu_{k}$. The assumption \eqref{eq:rvassumption} is satisfied, for example, by standard Gaussian random variables, standard Bernoulli random variables, or any random variables with compactly supported distributions. 
\end{remark}

We emphasize that given a pair of \emph{radially symmetric} functions $(f_0, f_1)$, the resulting random data $(f_0^\omega, f_1^\omega)$ is \emph{not} radially symmetric. This is due to the fact that multiplying the non-radial symbols of the Fourier projections $P_k$ as in~\eqref{equ:unit_scale_proj} by independent Gaussian random variables destroys the radial symmetry upon summation. Nonetheless, we are able to exploit the underlying radial symmetry of the pair $(f_0, f_1)$ in the proof of a key large deviation estimate for a weighted space-time norm of the free wave evolution of the associated random data, see Proposition~\ref{prop:large_deviation_L2Linfty}. A crucial ingredient in this proof, which is the content of Lemma~\ref{lem:radialish_sobolev}, is the fact that for radially symmetric $f \in H^s_x(\bR^4)$, one has
\begin{equation} \label{equ:radialish_sobolev_intro}
 \biggl\| |x|^{\frac{3}{2}} \Bigl( \sum_{k \in \bZ^4} \bigl| P_k f(x) \bigr|^2 \Bigr)^{\frac{1}{2}} \biggr\|_{L^\infty_x(\bR^4)} \lesssim \| f \|_{H^s_x(\bR^4)}
\end{equation}
for any $s > 0$. The estimate \eqref{equ:radialish_sobolev_intro} demonstrates that the square-function corresponding to unit-scale projections of radially symmetric functions, given by
\[
 \Bigl( \sum_{k \in \bZ^4} \bigl| P_k f(x) \bigr|^2 \Bigr)^{\frac{1}{2}}
\]
satisfies a radial Sobolev-type estimate on $\bR^4$. One should also compare~\eqref{equ:radialish_sobolev_intro} to the usual radial Sobolev estimate in four space dimensions
\[
 \bigl\| |x|^{\frac{3}{2}} f \bigr\|_{L^\infty_x(\bR^4)} \lesssim \| f \|_{H^1_x(\bR^4)}.
\]

\subsection{Statement of main results and strategy of proof}

Let $(f_0, f_1) \in H^s_x(\bR^4) \times H^{s-1}_x(\bR^4)$ be a pair of real-valued, radially symmetric functions of super-critical regularity $s < 1$ and denote by $(f_0^\omega, f_1^\omega)$ the associated random initial data defined in~\eqref{equ:bighsrandomization}. In order to study the asymptotic behavior of the solution $u(t)$ to the energy-critical cubic nonlinear wave equation~\eqref{equ:ivp} with random data $(f_0^\omega, f_1^\omega)$, we express $u(t)$ as 
\begin{equation}
 u(t) = S(t)(P_{>4} f_0^\omega, P_{>4} f_1^\omega) + v(t)
\end{equation}
and we investigate the asymptotic behavior of the nonlinear component $v(t)$ satisfying the difference equation
\begin{equation} \label{equ:forced_cubic_random}
 \left\{ \begin{aligned}
  -\partial_t^2 v + \Delta v &= \bigl( S(t)( P_{> 4} f_0^\omega, P_{>4} f_1^\omega ) + v \bigr)^3 \text{ on } \bR \times \bR^4, \\
  (v, \partial_t v)|_{t=0} &= (P_{\leq 4} f_0^\omega , P_{\leq 4} f_1^\omega ) \in \dot{H}^1_x(\bR^4) \times L^2_x(\bR^4).
 \end{aligned} \right.
\end{equation}
Here, the projections $P_{\leq 4}$ and $P_{>4}$ are the usual dyadic Littlewood-Paley projections defined in Section~\ref{sec:prelim}. We remark that we incorporate the low-frequency smooth part $(P_{\leq 4} f_0^\omega, P_{\leq 4} f_1^\omega)$ of the random data into the initial data for the Cauchy problem~\eqref{equ:forced_cubic_random} for the nonlinear component~$v(t)$ only to avoid certain technical complications with the zero frequency in the proof of the large deviation estimate in Proposition~\ref{prop:large_deviation_L2Linfty}.

\medskip 

Thus, we are led to generalize our study of the asymptotic behavior of the solution to~\eqref{equ:forced_cubic_random} to that of the asymptotic behavior of solutions to the forced cubic nonlinear wave equation
\begin{equation} \label{equ:forced_cubic}
 \left\{ \begin{aligned}
          -\partial_t^2 v + \Delta v &= (F+v)^3 \text{ on } \bR \times \bR^4, \\
          (v, \partial_t v)|_{t=0} &= (v_0, v_1) \in \dot{H}^1_x(\bR^4) \times L^2_x(\bR^4)
         \end{aligned} \right.
\end{equation}
for \emph{arbitrary} real-valued initial data $(v_0, v_1)$ in the energy class, and forcing term $F \colon \bR \times \bR^4 \to \bR$ satisfying suitable space-time integrability properties.

\medskip 

The standard local Cauchy theory for the cubic nonlinear wave equation on $\bR^4$ adapts readily to~\eqref{equ:forced_cubic}. In particular, one can show that for initial data $(v_0, v_1) \in \dot{H}^1_x(\bR^4) \times L^2_x(\bR^4)$ and $F \in L^3_{t, loc} L^6_x(\bR \times \bR^4)$, there exists a unique solution 
\begin{equation}
 (v, \partial_t v) \in C \bigl( I_{\ast}; \dot{H}^1_x(\bR^4) \bigr) \cap L_{t, loc}^3 L_x^6(I_{\ast} \times \bR^4) \times C \bigl( I_{\ast}; L^2_x(\bR^4) \bigr)
\end{equation}
to the Cauchy problem~\eqref{equ:forced_cubic} with maximal time interval of existence $I_\ast = (T_-, T_+)$. Furthermore, we have the finite time blowup criterion: 
\begin{equation}
 T_+ < \infty \qquad \Longrightarrow \qquad \|v\|_{L_t^3 L_x^6([0, T_+) \times \bR^4)} = + \infty
\end{equation}
with an analogous statement if $T_- > -\infty$. Finally, if the forcing term satisfies the stronger condition $F \in L^3_t L^6_x(\bR \times \bR^4)$, then a global solution $v(t)$ to~\eqref{equ:forced_cubic} scatters to free waves if $\| v \|_{L^3_t L^6_x(\bR \times \bR^4)} < \infty$. Correspondingly, we refer to the $L_t^3 L_x^6(\bR \times \bR^4)$ norm as a scattering norm.

\medskip 

The following observation reduces the complete analysis of the global and asymptotic behavior of solutions to~\eqref{equ:forced_cubic} to proving uniform-in-time a priori bounds on the energy of these solutions on their maximal time intervals of existence. A fundamental difficulty in the study of the forced cubic nonlinear wave equation, however, is that it is no longer a Hamiltonian equation so there is no conserved energy functional, and obtaining such energy bounds is therefore non-trivial.

\begin{theorem} \label{thm:conditional_scattering}
 There exists a non-decreasing function $K \colon [0,\infty) \to [0, \infty)$ with the following property. Let $(v_0, v_1) \in \dot{H}^1_x(\bR^4) \times L^2_x(\bR^4)$ and $F \in L^3_t L^{6}_x(\bR \times \bR^4)$. Let $v(t)$ be a solution to \eqref{equ:forced_cubic} defined on its maximal time interval of existence $I_\ast$. Suppose in addition that
 \begin{equation} \label{equ:energy_hypothesis}
  M := \sup_{t \in I_\ast} \, E(v(t)) < \infty,
 \end{equation}
 where
 \[
  E(v(t)) = \int_{\bR^4} \frac{1}{2} |\nabla_x v(t)|^2 + \frac{1}{2} |\partial_t v(t)|^2 + \frac{1}{4} |v(t)|^4 \, dx.
 \]
 Then $I_\ast = \bR$, that is $v(t)$ is globally defined, and it holds that
 \begin{equation} \label{equ:conditional_spacetime_bound}
  \|v\|_{L^3_t L^{6}_x(\bR\times\bR^4)} \leq C \|F\|_{L^3_t L^{6}_x(\bR\times\bR^4)} \bigl( K(M) + 1 \bigr) \exp \bigl( C \,K(M)^3 \bigr)
 \end{equation}
 for some absolute constant $C > 0$. In particular, the solution $v(t)$ scatters to free waves as $t \to \pm \infty$ in the sense that there exist initial data $(v_0^{\pm}, v_1^{\pm}) \in \dot{H}^1_x(\bR^4) \times L^2_x(\bR^4)$ such that
 \[
  \lim_{t \to \pm \infty} \, \bigl\| \nabla_{t,x} \bigl( v(t) - v_L^{\pm}(t) \bigr) \bigr\|_{L^2_x(\bR^4)} = 0,
 \]
 where 
 \[
  v_L^{\pm}(t) = \cos(t |\nabla|) v_0^{\pm} + \frac{\sin(t|\nabla|)}{|\nabla|} v_1^{\pm}.
 \]
\end{theorem}

An important ingredient in the proof of Theorem~\ref{thm:conditional_scattering} are the a priori bounds on the global scattering norms of solutions to the defocusing cubic nonlinear wave equation on $\bR^4$ coming from the classical work of Bahouri-G\'erard~\cite{Bahouri_Gerard}. In particular, we exploit the fact that these a priori bounds depend only on the energy of the initial data. Using the divisibility of the $L^3_t L^6_x(\bR \times \bR^4)$ norm of the forcing term $F$, the proof of Theorem~\ref{thm:conditional_scattering} follows from these a priori bounds and an iterative application of a suitable perturbation theory which enables us to compare solutions of the forced equation ~\eqref{equ:forced_cubic} to those of the ``unforced'' cubic nonlinear wave equation on $\bR^4$. 

\begin{remark}
 We emphasize that the hypotheses of Theorem~\ref{thm:conditional_scattering} do not include any assumptions of radial symmetry on the initial data or the forcing term $F$, or any assumptions on the Sobolev regularity for the forcing term $F$. An analogous result holds with an identical proof for the forced quintic nonlinear wave equation on $\bR^3$ up to replacing the $L^3_t L^6_x(\bR \times \bR^4)$ by the $L^5_t L^{10}_x(\bR \times \bR^3)$ norm.
\end{remark}

\begin{remark} \label{rem:get_global_without_scattering}
 If one is only interested in conditional global existence without scattering for the forced cubic nonlinear wave equation~\eqref{equ:forced_cubic}, then the proof of Theorem~\ref{thm:conditional_scattering} implies that it suffices to make the weaker assumption $F \in L^3_{t, loc} L^6_x(\bR \times \bR^4)$. Finally, if one assumes that
 \begin{equation} \label{equ:assumption_on_F_global_without_scattering}
  F \in L^3_{t,loc} L^6_x(\bR \times \bR^4) \cap L^1_{t,loc} L^\infty_x(\bR \times \bR^4),
 \end{equation}
 then one may infer from the Gronwall-type estimate of Burq-Tzvetkov~\cite[Proposition 2.2]{BT4} that the energy of the solution to the forced cubic nonlinear wave equation~\eqref{equ:forced_cubic} cannot blow up in finite time. A straightforward adaptation of the arguments from~\cite{BT4} to this setting shows that for arbitrary real-valued $(f_0, f_1) \in H^s_x(\bR^4) \times H^{s-1}_x(\bR^4)$, $0 < s < 1$, the free wave evolution of the associated random data almost surely satisfies~\eqref{equ:assumption_on_F_global_without_scattering}, thus recovering the almost sure global well-posedness result of Pocovnicu~\cite[Theorem 1.3]{Pocovnicu} for the energy-critical defocusing nonlinear wave equation on $\bR^4$ by taking $F \equiv S(t)(f_0^\omega, f_1^\omega)$.
\end{remark}

\begin{remark} \label{rmk:pocovnicu_compare}
While the proof of Theorem~\ref{thm:conditional_scattering} and the discussion in Remark~\ref{rem:get_global_without_scattering} may be reminiscent of the proof of almost sure global well-posedness (without scattering) for the defocusing cubic nonlinear wave equation on $\bR^4$ by Pocovnicu~\cite[Theorem 1.3]{Pocovnicu}, there are important differences. While Pocovnicu's proof also relies on perturbation theory and the a priori bounds due to Bahouri-G\'erard~\cite{Bahouri_Gerard}, the application of the perturbation theory within the proof of a key ``good'' local well-posedness result~\cite[Proposition 4.3]{Pocovnicu} requires a certain smallness condition on the $L^3_t L^6_x$ norm of the forcing term on short time intervals in terms of the actual length of these time intervals, see (4.5) in \cite{Pocovnicu}. Therefore, the method in~\cite{Pocovnicu} can only yield space-time bounds on arbitrarily long, but finite, time intervals and hence falls short of being an approach to prove scattering. In contrast, in our proof of Theorem~\ref{thm:conditional_scattering}, the requisite smallness condition for the perturbation theory is provided by the standard divisibility of the $L^3_t L^6_x(\bR \times \bR^4)$ norm of the forcing term $F$ and does not depend on its specific structure.
\end{remark}

In light of the previous theorem, in order to establish global existence and scattering for the forced cubic nonlinear wave equation~\eqref{equ:forced_cubic}, we focus on proving the a priori energy bound~\eqref{equ:energy_hypothesis} for solutions of~\eqref{equ:forced_cubic}. In our next theorem, we present some sufficient conditions on the forcing term $F$ in order to establish these uniform-in-time energy bounds.

\begin{theorem} \label{thm:scattering}
 Let $(v_0, v_1) \in \dot{H}^1_x(\bR^4) \times L^2_x(\bR^4)$. Assume that
 \begin{align}\label{equ:driving_bds}
 F \in L^3_t L^{6}_x(\bR \times \bR^4) \qquad \textup{and}  \qquad |x|^{\frac{1}{2}} F \in L^2_t L^{\infty}_x(\bR \times \bR^4).
 \end{align}
 Let $v(t)$ be a solution to \eqref{equ:forced_cubic} defined on its maximal time interval of existence $I_\ast$. Then we have 
 \[
  \sup_{t \in I_\ast} \, E(v(t)) \leq C \exp \Bigl( C \bigl( \|F\|_{L^3_t L^6_x(\bR\times\bR^4)}^3 + \bigl\| |x|^{\frac{1}{2}} F \bigr\|_{L^2_t L^\infty_x(\bR\times\bR^4)}^2 \bigr) \Bigr) ( E(v(0)) + 1 )
 \]
 for some absolute constant $C > 0$. It therefore holds that $I_\ast = \bR$, that is $v(t)$ exists globally in time, and the solution $v(t)$ scatters to free waves as $t \to \pm \infty$.
\end{theorem}

The main idea of the proof of Theorem~\ref{thm:scattering} is to combine an approximate Morawetz identity for the forced cubic nonlinear wave equation~\eqref{equ:forced_cubic} with a Gronwall-type estimate due to Burq-Tzvetkov~\cite[Proposition 2.2]{BT4} in order to achieve uniform-in-time control on the energy of solutions to~\eqref{equ:forced_cubic}.

\begin{remark}
 Under the assumption that the forcing term satisfies
 \[
  F \in L^3_t L^6_x(\bR\times\bR^4) \cap L^1_t L^\infty_x(\bR\times\bR^4),
 \] 
 the Gronwall-type argument due to Burq-Tzvetkov~\cite[Proposition 2.2]{BT4} would imply uniform-in-time bounds on the energy of solutions to~\eqref{equ:forced_cubic}. However, proving global-in-time $L^1_t L^\infty_x(\bR \times \bR^4)$ bounds for the free wave evolution of randomized initial data seems to be out of reach.
\end{remark}

Finally, we are in a position to state our main almost sure scattering result for the energy-critical defocusing nonlinear wave equation on $\bR^4$.

\begin{theorem}\label{thm:random_scattering}
Let $\frac{1}{2} < s < 1$. For real-valued radially symmetric $(f_0, f_1) \in H^s_x(\bR^4) \times H^{s-1}_x(\bR^4)$, let $(f_0^\omega, f_1^\omega)$ be the randomized initial data defined in \eqref{equ:bighsrandomization}. Then for almost every $\omega \in \Omega$, there exists a unique global solution 
 \begin{equation} \label{equ:solution_main_theorem}
  (u, \partial_t u) \in \bigl( S(t)(f_0^\omega, f_1^\omega), \partial_t S(t)(f_0^\omega, f_1^\omega) \bigr) + C\bigl(\bR; \dot{H}^1_x(\bR^4) \times L^2_x(\bR^4)\bigr)
 \end{equation}
 to the energy-critical defocusing nonlinear wave equation
 \begin{equation} \label{equ:nlw_main_theorem}
  \left\{ \begin{aligned}
   -\partial_t^2 u + \Delta u &= u^3 \text{ on } \bR \times \bR^4, \\
   (u, \partial_t u)|_{t=0} &= (f_0^\omega, f_1^\omega),
  \end{aligned} \right.
 \end{equation}
 which scatters to free waves as $t \to \pm \infty$ in the sense that there exist data $(v_0^{\pm}, v_1^{\pm}) \in \dot{H}^1_x(\bR^4) \times L^2_x(\bR^4)$ such that 
 \[
  \lim_{t \to \pm \infty} \, \bigl\| \nabla_{t,x} \bigl( u(t) - S(t)(f_0^\omega, f_1^\omega) - v_L^{\pm}(t) \bigr) \bigr\|_{L^2_x(\bR^4)} = 0,
 \]
 where
 \[
  v_L^{\pm}(t) = \cos(t |\nabla|) v_0 + \frac{\sin(t|\nabla|)}{|\nabla|} v_1.
 \]
\end{theorem}

\begin{remark} \label{rem:uniqueness}
In the statement of Theorem~\ref{thm:random_scattering}, uniqueness holds in the sense that upon writing 
 \[
  (u, \partial_t u) = \bigl( S(t)( P_{>4} f_0^\omega, P_{>4} f_1^\omega), \partial_t S(t)( P_{>4} f_0^\omega, P_{>4} f_1^\omega) \bigr) + (v, \partial_t v),
 \]
 there exists a unique global solution 
 \[
  (v, \partial_t v) \in C \bigl(\bR; \dot{H}^1_x(\bR^3)\bigr) \cap L^{3}_{t, loc} L^{6}_x(\bR\times\bR^4) \times C\bigl(\bR; L^2_x(\bR^4)\bigr)
 \]
 to the forced cubic nonlinear wave equation
 \begin{equation} 
  \left\{ \begin{aligned}
   -\partial_t^2 v + \Delta v &= \bigl( S(t)(P_{>4}  f_0^\omega, P_{>4} f_1^\omega) + v \bigr)^{3} \text{ on } \bR \times \bR^4, \\
   (v, \partial_t v)|_{t=0} &=  (P_{\leq 4}  f_0^\omega, P_{\leq 4} f_1^\omega),
  \end{aligned} \right.
 \end{equation}
 where once again, the projections $P_{\leq 4}$ and $P_{> 4}$ are the usual dyadic Littlewood-Paley projections, defined in Section \ref{sec:prelim}.
\end{remark}

In light of Theorem~\ref{thm:scattering} and our earlier discussion about the difference equation~\eqref{equ:forced_cubic_random}, the proof of Theorem~\ref{thm:random_scattering} reduces to showing that the forcing term $F = S(t)(P_{>4} f_0^\omega, P_{>4} f_1^\omega)$ satisfies the space-time bounds~\eqref{equ:driving_bds} almost surely. Here the main challenge is to establish that almost surely,
\[
 \bigl\| |x|^{\frac{1}{2}} S(t)(P_{>4} f_0^\omega, P_{>4} f_1^\omega) \bigr\|_{L^2_t L^\infty_x(\bR \times \bR^4)} < \infty.
\]
The proof of this bound follows from a large deviation estimate which hinges on the radial symmetry of the pair $(f_0, f_1)$, see Proposition~\ref{prop:large_deviation_L2Linfty}. In contrast, no radial symmetry assumptions on $(f_0, f_1)$ are required to show that almost surely
\[
 \bigl\| S(t)(P_{>4} f_0^\omega, P_{>4} f_1^\omega) \bigr\|_{L^3_t L^6_x(\bR \times \bR^4)} < \infty,
\]
which follows from methods already implemented by the second and third authors~\cite[Theorem~1.6]{LM}. We improve upon these estimates in the present work through the use of the Klainerman-Tataru refinement of Strichartz estimates, see Proposition~\ref{prop:large_deviation_L3L6}.

\medskip 

As pointed out earlier, although the pair $(f_0, f_1)$ is radially symmetric, the resulting randomized data $(f_0^\omega, f_1^\omega)$ is not radially symmetric. One may be inclined to use a different type of randomization that preserves the radial symmetry of the pair $(f_0, f_1)$ upon randomization. However, our ultimate goal is to obtain the result of Theorem~\ref{thm:random_scattering} for unit-scale randomized non-radial initial data of super-critical regularity, and so we use the unit-scale randomization procedure~\eqref{equ:bighsrandomization}, which has proven to be an effective way to randomize initial data on Euclidean space.

\medskip

Theorem~\ref{thm:random_scattering} is the first almost sure scattering result for an energy-critical problem with scaling super-critical initial data. We note that for energy sub-critical nonlinear wave equations on~$\bR^3$, almost sure scattering results were obtained by de~Suzzoni~\cite{Suzzoni1, Suzzoni2}. In these works almost sure global well-posedness is first established for related energy sub-critical nonlinear wave equations posed on the three-dimensional unit sphere following the methods of Burq-Tzvetkov~\cite{BT1, BT2, BT4}. Using the Penrose transform, these solutions are then mapped to scattering solutions of energy sub-critical nonlinear wave equations on $\bR^3$ with super-critical random initial data belonging to weighted Sobolev spaces.

\subsection*{Notation} We denote by $C > 0$ an absolute constant that depends only on fixed parameters and whose value may change from line to line. We write $X \lesssim Y$ if $X \leq C Y$ and we use the notation $X \sim Y$ to indicate that $X \lesssim Y \lesssim X$. Moreover, we write $X \ll Y$ if the implicit constant should be regarded as small.

\subsection*{Acknowledgements} The authors are grateful to Carlos Kenig, Joachim Krieger, Chris Sogge, and Gigliola Staffilani for useful discussions.

\section{Deterministic preliminaries} \label{sec:prelim}

\subsection{Harmonic analysis}

We begin by introducing the usual dyadic Littlewood-Paley projections. Let $\varphi \in C_c^\infty(\bR^4)$ be a radial smooth bump function satisfying $\varphi(\xi) = 1$ for $|\xi| \leq 1$ and $\varphi(\xi) = 0$ for $|\xi| > 2$. For $N \in 2^{\bZ}$ we define 
\begin{align*}
 \widehat{P_{\leq N} f}(\xi) &:= \varphi(\xi/N) \hat{f}(\xi), \\
 \widehat{P_{> N} f}(\xi) &:= \bigl( 1 - \varphi(\xi/N) \bigr) \hat{f}(\xi), \\
 \widehat{P_N f}(\xi) &:= \bigl( \varphi(\xi/N) - \varphi(2\xi/N) \bigr) \hat{f}(\xi).
\end{align*}
Moreover, we introduce the fattened Littlewood-Paley projection $\widetilde{P}_N := P_{\leq 8 N} - P_{\leq N/8}$. These projections satisfy the usual Bernstein estimates.
\begin{lemma} 
 For $1 \leq p \leq q \leq \infty$ and $s \geq 0$, it holds that
 \begin{align*}
  \bigl\| P_N f \bigr\|_{L^q_x(\bR^4)} &\lesssim N^{\frac{4}{p} - \frac{4}{q}} \bigl\| P_N f \bigr\|_{L^p_x(\bR^4)}, \\
  \bigl\| P_{\leq N} f \bigr\|_{L^q_x(\bR^4)} &\lesssim N^{\frac{4}{p} - \frac{4}{q}} \bigl\| P_{\leq N} f \bigr\|_{L^p_x(\bR^4)}, \\
  \bigl\| |\nabla|^{\pm s} P_N f \bigr\|_{L^p_x(\bR^4)} &\sim N^{\pm s} \bigl\| P_N f \bigr\|_{L^p_x(\bR^4)}, \\
  \bigl\| \langle \nabla \rangle^s P_N f \bigr\|_{L^p_x(\bR^4)} &\lesssim (1 + N^s) \bigl\| P_N f \bigr\|_{L^p_x(\bR^4)}.
 \end{align*}
\end{lemma}
We denote dyadic Littlewood-Paley projections $P_N$ using upper-case letters for the frequency, while we denote the unit-scale projections of \eqref{equ:unit_scale_proj} by $P_k$.

\medskip 

For the unit-scale projections, we have the following square-function estimate which should be compared with the usual radial Sobolev embedding in $\bR^4$. 
\begin{lemma} \label{lem:radialish_sobolev}
 Let $s > 0$ and let $f \in H^s_x(\bR^4)$ be radially symmetric. Then for $P_k$ as defined in \eqref{equ:unit_scale_proj}, it holds that
 \begin{align}
  \biggl\| |x|^{\frac{3}{2}} \Bigl( \sum_{k \in \bZ^4} \bigl| P_k f(x) \bigr|^2 \Bigr)^{\frac{1}{2}} \biggr\|_{L^\infty_x(\bR^4)} \lesssim \| f \|_{H^s_x(\bR^4)}.
 \end{align}
\end{lemma}
\begin{proof}
It suffices to consider the more delicate case of summing over all lattice points $k \in \bZ^4$ with $|k| \gg 1$. Fix $x \in \bR^4$. We may assume that $|x| \gg 1$ and that we have chosen coordinates so that 
\[
 x = (0, 0, 0, |x|) = |x| e_4
\]
with $e_4 = (0,0,0,1)$. Then it holds that
\begin{align*}
 P_k f(x) &= \int_{\bR^4_{\xi}} e^{i x \cdot \xi} \psi_k(\xi) \hat{f}(\xi) \, d\xi \\
 &= \int_0^{\infty} \int_0^{\pi} \int_{\bS^2} e^{i |x| \rho \cos(\theta)} \psi_k\bigl( \xi(\rho, \theta, \omega) \bigr) \hat{f}(\rho) \rho^3 \sin^2(\theta) \, d\omega \, d\theta \, d\rho \\
 &= \int_0^{\infty} \biggl( \int_0^{\pi} \int_{\bS^2} e^{i |x| \rho \cos(\theta)} \sin^2(\theta) \psi_k\bigl( \xi(\rho, \theta, \omega) \bigr) d\omega \, d\theta \biggr) \hat{f}(\rho) \rho^3 \, d\rho,
\end{align*}
where we are using spherical coordinates 
 \[
  0 \leq \rho < \infty, \quad 0 \leq \theta \leq \pi, \quad \omega \in \bS^2 \hookrightarrow \bR^3
 \]
 such that
 \[
  \xi = (\xi_1, \xi_2, \xi_3, \xi_4) = ( \rho \omega \sin(\theta), \rho \cos(\theta) ).
 \]
Integrating by parts in $\theta$ we find 
 \begin{equation} \label{equ:int_by_parts_once}
  \begin{aligned}
   &\int_0^{\pi} \int_{\bS^2} e^{i |x| \rho \cos(\theta)} \sin^2(\theta) \psi_k\bigl( \xi(\rho, \theta, \omega) \bigr) d\omega \, d\theta \\
   &= \int_0^{\pi} \int_{\bS^2} \frac{1}{- i |x| \rho \sin(\theta)} \frac{\partial}{\partial \theta} \Bigl( e^{i |x| \rho \cos(\theta)} \Bigr) \sin^2(\theta) \psi_k\bigl( \xi(\rho, \theta, \omega) \bigr) \, d\omega \, d\theta \\
   &= \frac{1}{i |x| \rho} \int_0^{\pi} \int_{\bS^2} e^{i |x| \rho \cos(\theta)} \cos(\theta) \psi_k\bigl( \xi(\rho, \theta, \omega) \bigr) \, d\omega \, d\theta \\
   &\quad + \frac{1}{i |x| \rho} \int_0^{\pi} \int_{\bS^2} e^{i |x| \rho \cos(\theta)} \sin(\theta) \frac{\partial}{\partial \theta} \Bigl( \psi_k\bigl( \xi(\rho, \theta, \omega) \bigr) \Bigr) \, d\omega \, d\theta \\
   &\equiv I + II.
  \end{aligned}
 \end{equation}
Next we observe that since $|k| \gg 1$, the bump function $\psi_k\bigl( \xi(\rho, \theta, \omega) \bigr)$ as well as its derivatives $(\nabla_{\xi}^j \psi_k)\bigl( \xi(\rho, \theta, \omega) \bigr)$, $j = 1, 2$, enforce a localization of the variable $\rho$ to an interval of size $\sim 1$ around~$|k|$, a localization of the variable $\theta$ to an interval of size $\sim \frac{1}{|k|}$ around an angle $\theta_k$ with 
 \[
  \sin(\theta_k) = \sqrt{1 - (e_4 \cdot \textstyle{\frac{k}{|k|})^2}},
 \]
 and a localization of the variable $\omega$ to a cap $\kappa(k) \subset \bS^2$ of surface area $|\kappa(k)| \lesssim \min \{ 1, \frac{1}{|k|^2 \sin^2(\theta_k)} \}$. Moreover, we have the following derivative bounds
 \begin{align}\label{deriv_bd_1}
  \Bigl| \frac{\partial}{\partial \theta} \Bigl( \psi_k\bigl( \xi(\rho, \theta, \omega) \bigr) \Bigr) \Bigr| \lesssim \rho \bigl| ( \nabla_\xi \psi_k ) \bigl( \xi(\rho, \theta, \omega) \bigr) \bigr|
 \end{align}
 and 
  \begin{align}\label{deriv_bd_2}
  \Bigl| \frac{\partial^2}{\partial \theta^2} \Bigl( \psi_k\bigl( \xi(\rho, \theta, \omega) \bigr) \Bigr) \Bigr| \lesssim \rho^2 \Bigl( \bigl| ( \nabla_\xi \psi_k ) \bigl( \xi(\rho, \theta, \omega) \bigr) \bigr| + \bigl| ( \nabla_\xi^2 \psi_k ) \bigl( \xi(\rho, \theta, \omega) \bigr) \bigr| \Bigr).
 \end{align}
 
We now return to estimating the terms $I$ and $II$. We begin with the term $I$. Here, a careful application of the van der Corput lemma together with the above observations reveals that 
\begin{equation} \label{equ:term1_bound}
 \bigl| I \bigr| \lesssim \bigl( |x| |k| \bigr)^{-\frac{3}{2}} \min \Bigl\{ 1, \frac{1}{|k|^2 \sin^2(\theta_k)} \Bigr\} \chi_{[|k|-1, |k|+1]}(\rho),
\end{equation}
where $\chi_{[|k|-1, |k|+1]}(\cdot)$ denotes a sharp cut-off to the interval $[|k|-1, |k|+1]$.
 
Next we treat the term $II$. On the one hand, we have from \eqref{deriv_bd_1} that
\begin{equation} \label{equ:term2_first_bound}
 \bigl| II \bigr| \lesssim \frac{1}{|x|} \frac{1}{|k|} \sin(\theta_k) \min \Bigl\{ 1, \frac{1}{|k|^2 \sin^2(\theta_k)} \Bigr\} \chi_{[|k|-1, |k|+1]}(\rho).
\end{equation}
On the other hand, integrating by parts in $\theta$ once more in the term $II$ and using \eqref{deriv_bd_2}, we find 
\begin{equation} \label{equ:term2_second_bound}
 \bigl| II \bigr| \lesssim \frac{1}{|x|^2} \frac{1}{|k|} \min \Bigl\{ 1, \frac{1}{|k|^2 \sin^2(\theta_k)} \Bigr\} \chi_{[|k|-1, |k|+1]}(\rho). 
\end{equation}
Interpolating between \eqref{equ:term2_first_bound} and \eqref{equ:term2_second_bound} yields
\begin{equation} \label{equ:term2_bound}
 \bigl| II \bigr| \lesssim \frac{1}{|x|^{\frac{3}{2}}} \frac{1}{|k|} \sin^{\frac{1}{2}}(\theta_k) \min \Bigl\{ 1, \frac{1}{|k|^2 \sin^2(\theta_k)} \Bigr\} \chi_{[|k|-1, |k|+1]}(\rho).
\end{equation}
Thus, combining the bounds \eqref{equ:term1_bound} and \eqref{equ:term2_bound} we obtain 
\begin{align*}
 &\biggl| \int_0^{\pi} \int_{\bS^2} e^{i |x| \rho \cos(\theta)} \sin^2(\theta) \psi_k\bigl( \xi(\rho, \theta, \omega) \bigr) d\omega \, d\theta \biggr| \\
 &\quad \quad \quad \lesssim \frac{1}{|x|^{\frac{3}{2}}} \frac{1}{|k|} \Bigl( \frac{1}{|k|^{\frac{1}{2}}} + \sin^{\frac{1}{2}}(\theta_k) \Bigr) \min \Bigl\{ 1, \frac{1}{|k|^2 \sin^2(\theta_k)} \Bigr\} \chi_{[|k|-1, |k|+1]}(\rho).
\end{align*}
Returning to our estimate for $P_k f(x)$, we conclude
\begin{align*}
 \bigl| P_k f(x) \bigr| \lesssim \frac{1}{|x|^{\frac{3}{2}}} \frac{1}{|k|} \Bigl( \frac{1}{|k|^{\frac{1}{2}}} + \sin^{\frac{1}{2}}(\theta_k) \Bigr) \min \Bigl\{ 1, \frac{1}{|k|^2 \sin^2(\theta_k)} \Bigr\}  \bigl\| \chi_{[|k|-1, |k|+1]}(\rho) \hat{f}(\rho) \rho^3  \bigr\|_{L^1_{\rho}}.
 \end{align*}
By H\"older's inequality this is bounded by
\begin{align*}
 \bigl| P_k f(x) \bigr| \lesssim \frac{1}{|x|^{\frac{3}{2}}} \Bigl( 1 + |k|^{\frac{1}{2}} \sin^{\frac{1}{2}}(\theta_k) \Bigr) \min \Bigl\{ 1, \frac{1}{|k|^2 \sin^2(\theta_k)} \Bigr\}  \bigl\| \chi_{[|k|-1, |k|+1]}(\rho) \hat{f}(\rho) \rho^{\frac{3}{2}} \bigr\|_{L^2_{\rho}},
\end{align*}
and thus, we arrive at the desired bound 
\begin{align*}
 |x|^3 \sum_{\substack{ k \in \bZ^4, \\ |k| \gg 1 }} \bigl| P_k f(x) \bigr|^2 &\lesssim \sum_{\substack{ k \in \bZ^4, \\ |k| \gg 1 }} \bigl( 1 + |k| \sin ( \theta_k ) \bigr) \min \Bigl\{ 1, \frac{1}{|k|^4 \sin^4(\theta_k)} \Bigr\} \bigl\| \chi_{[|k|-1, |k|+1]}(\rho) \hat{f}(\rho) \rho^{\frac{3}{2}} \bigr\|_{L^2_{\rho}}^2 \\
 &\lesssim \sum_{\substack{ j \in \bN, \\ j \gg 1}} \log(j) \bigl\| \chi_{[j-10, j+10]}(\rho) \hat{f}(\rho) \rho^{\frac{3}{2}} \bigr\|_{L^2_{\rho}}^2 \\
 &\lesssim \bigl\| f \bigr\|_{H^s_x(\bR^4)}^2. \qedhere
\end{align*}
\end{proof}

\subsection{Strichartz estimates}

We recall the Strichartz estimates for the linear wave equation on~$\bR^{1+4}$,
\begin{equation} \label{equ:linear_wave}
 \left\{ \begin{aligned}
  -\partial_t^2 u + \Delta u &= h, \\
  (u, \partial_t u)|_{t=0} &= (u_0, u_1).
 \end{aligned} \right.
\end{equation}
Let $2 \leq q \leq \infty$ and $2 \leq r < \infty$. We say that the pair $(q,r)$ is \emph{admissible} if
\begin{equation} \label{equ:admissible_condition}
 \frac{1}{q} + \frac{3}{2r} \leq \frac{3}{4}.
\end{equation}
If equality holds in \eqref{equ:admissible_condition} we say that $(q,r)$ is \emph{sharp admissible}.
\begin{proposition}[\protect{Strichartz estimates; \cite{Strichartz}, \cite{Pecher}, \cite{GV2}, \cite{KeelTao}}] \label{prop:strichartz_estimates}
 Let $(u_0, u_1) \in \dot{H}^\gamma_x(\bR^4) \times \dot{H}^{\gamma - 1}_x(\bR^4)$ for some~$\gamma > 0$ and let $u$ be a solution to the linear wave equation~\eqref{equ:linear_wave} on a time interval $I \ni 0$. Suppose $(q,r)$ and $(\tilde{q}, \tilde{r})$ are admissible pairs satisfying the scaling condition
 \[
  \frac{1}{q} + \frac{4}{r} = 2 - \gamma = \frac{1}{\tilde{q}'} + \frac{4}{\tilde{r}'} - 2.
 \]
 Then we have 
 \begin{equation*}
  \|u\|_{L^q_t L^r_x(I \times \bR^4)} \lesssim \|u_0\|_{\dot{H}^\gamma_x(\bR^4)} + \|u_1\|_{\dot{H}^{\gamma-1}_x(\bR^4)} + \|h\|_{L^{\tilde{q}'}_t L^{\tilde{r}'}_x(I\times\bR^4)}.
 \end{equation*}
\end{proposition}

It is well-known that for radially symmetric solutions to~\eqref{equ:linear_wave}, a larger range of Strichartz estimates is available, see for instance \cite{Sterbenz} and \cite{Fang_Wang}. In the proof of the key large deviation estimate for the free wave evolution of randomized radially symmetric data in Proposition~\ref{prop:large_deviation_L2Linfty} we will crucially exploit this fact. We say that the pair $(q,r)$ is \emph{radial-admissible} if
\begin{equation*}
 \frac{1}{q} + \frac{3}{r} < \frac{3}{2}.
\end{equation*}
\begin{proposition} [\protect{Radial Strichartz estimates; \cite{Sterbenz}, \cite{Fang_Wang}}] \label{prop:radial_strichartz_estimates}
 Let $(u_0, u_1) \in \dot{H}^\gamma_x(\bR^4) \times \dot{H}^{\gamma - 1}_x(\bR^4)$ be radially symmetric for some $\gamma > 0$ and let $u$ be a radially symmetric solution to the linear wave equation~\eqref{equ:linear_wave} (with $h=0$) on some time interval $I \ni 0$. Suppose $(q,r)$ is a radial-admissible pair satisfying the scaling condition
 \begin{align} \label{equ:scaling}
  \frac{1}{q} + \frac{4}{r} = 2 - \gamma.
 \end{align}
 Then we have 
 \begin{equation*}
  \|u\|_{L^q_t L^r_x(I \times \bR^4)} \lesssim \|u_0\|_{\dot{H}^\gamma_x(\bR^4)} + \|u_1\|_{\dot{H}^{\gamma-1}_x(\bR^4)}.
 \end{equation*}
\end{proposition}
In order to improve the range of Sobolev regularities for which our almost sure scattering result from Theorem~\ref{thm:random_scattering} holds, we exploit the refined Strichartz estimates of Klainerman-Tataru~\cite[Appendix A]{KlTa}. We first state these estimates in their most general form in four space dimensions.
\begin{proposition} [\protect{Klainerman-Tataru refinement of Strichartz estimates; \cite{KlTa}}] \label{prop:kt_strichart}
 Let $\ell, m \in \bZ$ with $\ell < m$ and let $(q,r)$ be an admissible pair. Let $\varphi_Q \in C_c^\infty(\bR^4)$ be a smooth bump function supported in a ball of radius $\sim 2^{\ell}$, which is located at distance $\sim 2^m$ from the origin. Denote by $P_Q$ the Fourier multiplier with symbol $\varphi_Q$. Then it holds that
 \begin{equation} \label{equ:strichartz_refined}
  \bigl\| e^{\pm i t |\nabla|} P_Q f \bigr\|_{L^q_t L^r_x(\bR \times \bR^4)} \lesssim 2^{(\frac{1}{2} - \frac{1}{r}) (\ell - m)} 2^{(2 - \frac{1}{q} - \frac{4}{r}) m} \| P_Q f \|_{L^2_x(\bR^4)}.
 \end{equation}
\end{proposition}
\begin{remark}
 The gain $2^{(\frac{1}{2} - \frac{1}{r}) (\ell - m)}$ in the refined Strichartz estimate~\eqref{equ:strichartz_refined} originates from a corresponding gain in the $L^1_x(\bR^4) \to L^\infty_x(\bR^4)$ dispersive estimate for the propagator $e^{\pm i t |\nabla|} P_Q$, see~\cite[(A.66)]{KlTa}. In~\cite{KlTa}, the estimate~\eqref{equ:strichartz_refined} is derived for all non-endpoint admissible pairs $(q,r) \neq (2,6)$ via the usual $TT^\ast$-argument. However, we note that one also obtains the refined Strichartz estimate~\eqref{equ:strichartz_refined} for the endpoint pair $(q,r) = (2,6)$ via the Keel-Tao argument~\cite{KeelTao}.
\end{remark}

In the following corollary, we reformulate the refined Strichartz estimates of Klainerman-Tataru into a convenient form for their application to the unit-scale frequency projections $P_k$ defined in~\eqref{equ:unit_scale_proj}.
\begin{corollary} \label{cor:KlTa_gain_applied_Pk}
 Let $(q,r)$ be an admissible pair. Then we have for all $k \in \bZ^4$ with $|k| > 4$ that
 \begin{equation} \label{equ:refined_cor}
  \bigl\| e^{\pm i t |\nabla|} P_k f \bigr\|_{L^q_t L^r_x(\bR \times \bR^4)} \lesssim \bigl\| |\nabla|^{\frac{1}{q}} P_k f \bigr\|_{L^2_x(\bR^4)}.
 \end{equation}
\end{corollary}
\begin{proof}
 Let $(q,r)$ be an admissible pair and let $\tilde{r} \leq r$ such that $(q, \tilde{r})$ is sharp admissible. Then by the unit-scale Bernstein estimate~\eqref{equ:unit_scale_bernstein}, the refined Strichartz estimate~\eqref{equ:strichartz_refined} for the pair $(q, \tilde{r})$, and the relation~\eqref{equ:admissible_condition} for the sharp admissible pair $(q,\tilde{r})$, we obtain that
 \begin{align*}
  \bigl\| e^{\pm i t |\nabla|} P_k f \bigr\|_{L^q_t L^r_x(\bR \times \bR^4)} &\lesssim \bigl\| e^{\pm i t |\nabla|} P_k f \bigr\|_{L^q_t L^{\tilde{r}}_x(\bR \times \bR^4)} \\
  &\lesssim 2^{-(\frac{1}{2} - \frac{1}{\tilde{r}}) \log_2(|k|)} 2^{(2 - \frac{1}{q} - \frac{4}{\tilde{r}}) \log_2(|k|)} \| P_k f \|_{L^2_x(\bR^4)} \\
  &\simeq 2^{\frac{1}{q} \log_2(|k|)} \| P_k f \|_{L^2_x(\bR^4)} \\
  &\lesssim \bigl\| |\nabla|^{\frac{1}{q}} P_k f \bigr\|_{L^2_x(\bR^4)}. \qedhere
 \end{align*}
\end{proof}

\section{Proof of Theorem \ref{thm:conditional_scattering}: Conditional scattering} \label{sec:cond_scat}

We begin by summarizing the local existence theory for the forced cubic nonlinear wave equation~\eqref{equ:forced_cubic} in the following lemma. Its proof is standard and uses Strichartz estimates for the wave equation.

\begin{lemma} \label{lem:local_theory}
 Let $(v_0, v_1) \in \dot{H}^1_x(\bR^4) \times L^2_x(\bR^4)$ and $F \in L^3_{t,loc} L^6_x(\bR\times\bR^4)$. Then there exists a unique solution 
 \[
  (v, \partial_t v) \in C\bigl( I_\ast; \dot{H}^1_x(\bR^4) \bigr) \cap L^3_{t,loc} L^{6}_x \bigl(I_\ast \times \bR^4 \bigr) \times C \bigl( I_\ast; L^2_x(\bR^4) \bigr)
 \]
 to the Cauchy problem~\eqref{equ:forced_cubic} with maximal time interval of existence $I_\ast = (T_-, T_+)$. Moreover, we have the standard finite time blowup criterion:
 \begin{equation} \label{equ:blowup_criterion}
  T_+ < \infty \qquad \Longrightarrow \qquad \|v\|_{L^3_t L^6_x([0,T_+)\times\bR^4)} = + \infty
 \end{equation}
 with an analogous statement if $T_- > -\infty$. If the forcing term satisfies the stronger condition $F \in L^3_t L^6_x(\bR\times\bR^4)$, then a global solution $v(t)$ to~\eqref{equ:forced_cubic} scatters to free waves as $t \to \pm \infty$ if $\|v\|_{L^3_t L^6_x(\bR\times\bR^4)} < \infty$.
\end{lemma}

The proof of Theorem~\ref{thm:conditional_scattering} crucially relies on the complete global well-posedness and scattering theory for the energy-critical defocusing nonlinear wave equation on~$\bR^4$ at energy regularity, established in the classical works of Struwe~\cite{Struwe}, Grillakis~\cite{Grillakis}, Ginibre-Soffer-Velo~\cite{Ginibre_Soffer_Velo}, Shatah-Struwe~\cite{Shatah_Struwe}, Bahouri-Shatah~\cite{Bahouri_Shatah}, Bahouri-G\'erard~\cite{Bahouri_Gerard}, Nakanishi~\cite{Nakanishi}, and Tao~\cite{Tao}.

\begin{theorem} \label{thm:cubic}
 There exists a non-decreasing function $K \colon [0,\infty) \to [0, \infty)$ with the following property. Let $(u_0, u_1) \in \dot{H}^1_x(\bR^4) \times L^2_x(\bR^4)$ and $t_0 \in \bR$. Then there exists a unique global solution
 \[
  (u, \partial_t u) \in C \bigl(\bR; \dot{H}^1_x(\bR^4) \bigr) \cap L^3_t L^{6}_x \bigl( \bR \times \bR^4 \bigr) \times C \bigl( \bR; L^2_x(\bR^4) \bigr)
 \]
 to the energy-critical defocusing nonlinear wave equation
 \begin{align} \label{equ:cubic}
  \left\{ \begin{aligned}
          -\partial_t^2 u + \Delta u &= u^3 \text{ on } \bR \times \bR^4, \\
          (u, \partial_t u)|_{t=t_0} &= (u_0, u_1) \in \dot{H}^1_x(\bR^4) \times L^2_x(\bR^4),
         \end{aligned} \right.
 \end{align}
 satisfying the a priori bound
 \[
  \|u\|_{L_t^3 L_x^{6}(\bR \times \bR^4 )} \leq K \bigl( E(u_0, u_1) \bigr),
 \]
 where
 \[
  E(u_0, u_1) = \int_{\bR^4} \frac{1}{2} |\nabla_x u_0|^2 + \frac{1}{2} |u_1|^2 + \frac{1}{4} |u_0|^4 \, dx.
 \]
 In particular, $u$ scatters to a free wave as $t \to \pm \infty$.
\end{theorem}

Finally, we need a suitable perturbation theory for the forced cubic nonlinear wave equation. Such perturbation results appeared first in the context of the three-dimensional energy-critical nonlinear Schr\"odinger equation in \cite{CKSTT08}. The proof of the following perturbation result is a straightforward modification of the standard formulation which typically involves nonlinearities of the form $u^3 + e$ for an additive error term $e$.

\begin{proposition} \label{prop:pert}
 Let $I = [t_0, t_1] \subset \bR$ be a compact time interval and let $(v_0, v_1) \in \dot{H}^1_x(\bR^4) \times L^2_x(\bR^4)$. Let $u(t)$ be the solution defined on $I \times \bR^4$ to the defocusing cubic nonlinear wave equation
 \begin{equation}
  \left\{ \begin{aligned}
   -\partial_t^2 u + \Delta u &= u^3 \text{ on } I \times \bR^4, \\
   (u, \partial_t u)|_{t=t_0} &= (v_0, v_1) \in \dot H^1_x(\bR^4) \times L^2_x(\bR^4).
  \end{aligned} \right.
 \end{equation}
 Suppose that
 \[
  \|u\|_{L_t^3L_x^{6}(I\times\bR^4)} \leq K
 \]
 for some $K > 0$. Then there exists $\varepsilon_0(K) := c \exp (- C K^3) > 0$ for some absolute constants $0 < c \ll 1$ and $C > 0$ such that if
 \[
  0 < \varepsilon \leq \varepsilon_0(K), 
 \] 
 then for any forcing term $F \in L_t^3 L_x^{6}(I \times \bR^4)$ with
 \[
  \|F\|_{L_t^3L_x^{6}(I\times\bR^4)} \leq \varepsilon,
 \]
 there exists a unique solution $v(t)$ on $I \times \bR^4$ to the forced cubic nonlinear wave equation
 \begin{equation}
  \left\{ \begin{aligned}
   -\partial_t^2 v + \Delta v &= (F+v)^3 \text{ on } I \times \bR^4, \\
   (v, \partial_t v)|_{t=t_0} &= (v_0,v_1) \in \dot{H}^1_x(\bR^4) \times L^2_x(\bR^4).
  \end{aligned} \right.
 \end{equation}
 Furthermore, we have the bound
 \begin{equation}
  \|\nabla_{t,x}( u -v)\|_{L_t^\infty L_x^2(I \times \bR^4)} + \| u -v\|_{L_t^3L_x^{6}(I \times \bR^4)} \lesssim \varepsilon \exp (C K^3).
 \end{equation}
\end{proposition}
 
We are now prepared to turn to the proof of Theorem \ref{thm:conditional_scattering}

\begin{proof}[Proof of Theorem \ref{thm:conditional_scattering}]
 Let $v(t)$ be the unique solution to the forced cubic nonlinear wave equation~\eqref{equ:forced_cubic} defined on its maximal interval of existence $I_\ast = (T_-, T_+)$ and satisfying the a priori energy bound~\eqref{equ:energy_hypothesis}. By Lemma~\ref{lem:local_theory}, it suffices to show that $\|v\|_{L^3_t L^6_x(I_\ast \times \bR^4)} < \infty$ in order to conclude that $v(t)$ exists globally and scatters to free waves as $t \to \pm \infty$. By time reversal symmetry it is enough to argue forward in time.
 
 We let $\varepsilon > 0$ be a small constant whose size will be fixed later depending only on the size~$M$ of the a priori energy bound~\eqref{equ:energy_hypothesis} for the solution $v(t)$. Then we partition the maximal forward time interval of existence $[0,T_+)$ of $v(t)$ into $N$ consecutive time intervals $I_j$, $j = 1, \ldots, N$, such that for each $j$, we have
 \[
  \| F \|_{L^3_t L^6_x(I_j \times \bR^4)} = \varepsilon,
 \]
 in particular it then holds that
 \[
  N \lesssim \frac{ \|F\|_{L^3_t L^6_x([0, T_+) \times\bR^4)}^3 }{\varepsilon^3}.
 \]
 We shall write $I_j \equiv [t_{j-1}, t_j]$ for $j = 1, \ldots, N-1$ with $t_0 = 0$ and $I_N \equiv [t_{N-1}, T_+)$. In order to show that $\|v\|_{L^3_t L^6_x([0,T_+)\times\bR^4)} < \infty$, we now bound the $L^3_t L^6_x(I_j\times\bR^4)$ norm of $v(t)$ on every interval~$I_j$, $j = 1, \ldots, N$. To this end, on each interval $I_j$, we compare the solution $v(t)$ of the forced cubic nonlinear wave equation
 \begin{equation}
  \left\{ \begin{aligned}
           -\partial_t^2 v + \Delta v &= (F+v)^3 \text{ on } I_j \times \bR^4, \\
	   (v, \partial_t v)|_{t=t_{j-1}} &= (v(t_{j-1}), \partial_t v(t_{j-1}))
          \end{aligned} \right.
 \end{equation}
 with the solution $u^{(j)}(t)$ of the standard defocusing cubic nonlinear wave equation with the same initial data at time $t=t_{j-1}$,
 \begin{equation}
  \left\{ \begin{aligned}
           -\partial_t^2 u^{(j)} + \Delta u^{(j)} &= (u^{(j)})^3 \text{ on } I_j \times \bR^4, \\
	   (u^{(j)}, \partial_t u^{(j)})|_{t=t_{j-1}} &= (v(t_{j-1}), \partial_t v(t_{j-1})),
          \end{aligned} \right.
 \end{equation}
 and infer the desired space-time bound for $v(t)$ on $I_j$ by applying the perturbation theory from Proposition~\ref{prop:pert}. We only treat the interval $I_1 \equiv [t_0, t_1]$ in more detail, as the other intervals are handled analogously. Consider the global solution $u^{(1)}(t)$ to the standard defocusing cubic nonlinear wave equation with initial data at time $t = t_0$ given by $(v(t_0), \partial_t v(t_0))$. Let $K \colon [0,\infty) \to [0,\infty)$ be the non-decreasing function from the statement of Theorem~\ref{thm:cubic}. Recalling that $M > 0$ denotes the size of the a priori energy bound~\eqref{equ:energy_hypothesis} for $v(t)$, by Theorem~\ref{thm:cubic} it holds that
 \[
  \| u^{(1)} \|_{L^3_t L^6_x(I_1 \times \bR^4)} \leq \| u^{(1)} \|_{L^3_t L^6_x(\bR\times\bR^4)} \leq K(M).
 \]
 By the perturbation result of Proposition \ref{prop:pert}, if we choose
 \[
  \varepsilon = \varepsilon_0 \bigl( K(M) \bigr)
 \]
 for $\varepsilon_0\bigl( K(M) \bigr)$ as in the statement of Proposition \ref{prop:pert}, then
 \begin{equation} \label{equ:finite_bd}
  \| u^{(1)} - v \|_{L_t^3 L_x^{6}(I_1 \times \bR^4)} \lesssim \varepsilon \exp \bigl( C K(M)^3 \bigr) \lesssim 1,
 \end{equation}
 from which we conclude that
 \begin{equation} 
  \| v \|_{L_t^3 L_x^{6}(I_1 \times \bR^4)} \leq \| u^{(1)} \|_{L_t^3 L_x^{6}(I_1 \times \bR^4)} + \| u^{(1)} - v \|_{L^3_t L^{6}_x(I_1 \times \bR^4)}  \lesssim K(M) + 1.
 \end{equation}
 Finally, we obtain the space-time bound~\eqref{equ:conditional_spacetime_bound} after summing up the bounds on the $L^3_t L^6_x$ norms of $v(t)$ on the finitely many time intervals $I_j$.
\end{proof}

\section{Proof of Theorem \ref{thm:scattering}: Energy bounds}

The main ingredient for the proof of Theorem~\ref{thm:scattering} is an approximate Morawetz estimate for the forced cubic nonlinear wave equation~\eqref{equ:forced_cubic} which we record in the following lemma.
\begin{lemma} \label{lem:morawetz}
 Let $I \subset \bR$ be a time interval and assume that $F \in L^3_t L^6_x(I \times \bR^4)$. Let $(v, \partial_t v) \in C \bigl( I; \dot{H}^1_x \times L^2_x \bigr)$ be a solution to~\eqref{equ:forced_cubic}. Then it holds that
 \begin{equation} \label{equ:morawetz}
  \begin{aligned}
   \int_I \int_{\bR^4} \frac{v^4}{|x|} \, dx \, dt \lesssim \| \nabla_{t,x} v \|_{L^\infty_t L^2_x(I \times \bR^4)}^2 + \| \nabla_x v \|_{L^{\infty}_t L^2_x(I \times \bR^4)} \bigl\| (F+v)^3 - v^3 \bigr\|_{L^1_t L^2_x(I \times \bR^4)}.  
  \end{aligned}
 \end{equation}
\end{lemma}
\begin{proof}
 One verifies that the solution $v(t)$ to~\eqref{equ:forced_cubic} satisfies the identity
 \begin{align*}
  \partial_t \int_{\bR^4} \Bigl( - \frac{x}{|x|}  \cdot \nabla_x v \partial_t v - \frac{3}{2} \frac{v}{|x|} \partial_t v \Bigr) \, dx &= \frac{3}{4} \int_{\bR^4} \frac{v^4}{|x|} \, dx + \frac{3}{4} \int_{\bR^4} \frac{v^2}{|x|^3} \, dx + \int_{\bR^4} \frac{ |\nabla_x v|^2 - ( {\textstyle \frac{x}{|x|} } \cdot \nabla_x v )^2 }{|x|} \, dx \\
  &\quad \quad + \int_{\bR^4} \Bigl( \frac{3}{2} \frac{v}{|x|} + \frac{x}{|x|} \cdot \nabla_x v \Bigr) \bigl( (F+v)^3 - v^3 \bigr) \, dx. 
 \end{align*}
 Estimate~\eqref{equ:morawetz} then follows from the fundamental theorem of calculus and Hardy's inequality. 
\end{proof}

We now prove Theorem~\ref{thm:scattering}. 

\begin{proof}[Proof of Theorem~\ref{thm:scattering}]
 Let $v(t)$ be the unique solution to the Cauchy problem~\eqref{equ:forced_cubic} with maximal time interval of existence $I_{\ast}$ satisfying
 \[
  (v, \partial_t v) \in C \bigl(I_\ast; \dot{H}^1_x(\bR^4)\bigr) \cap L^3_{t,loc} L^6_x(I_\ast \times \bR^4) \times C \bigl( I_\ast; L^2_x(\bR^4) \bigr).
 \]
By Theorem~\ref{thm:conditional_scattering}, in order to conclude global existence and scattering, it suffices to show that
 \[
  \sup_{t \in I_\ast} E(v(t)) < \infty,
 \]
 where we recall that
 \[
  E(v(t)) := \int_{\bR^4} \frac{1}{2} |\nabla_x v(t)|^2 + \frac{1}{2} |\partial_t v(t)|^2 + \frac{1}{4} |v(t)|^4 \, dx.
 \]
 To this end we compute  
 \begin{equation} \label{equ:energy_differentiated}
  \partial_t E(v(t)) = \int_{\bR^4} \partial_t v \bigl( \partial_t^2 v - \Delta v + v^3 \bigr) \, dx = - \int_{\bR^4} \partial_t v \bigl( (F+v)^3 - v^3 \bigr) \, dx.
 \end{equation}
 Then we define for $T > 0$ with $T \in I_\ast$ the quantities
 \begin{align*}
  A(T) &:= \int_0^T | \partial_t E( v(t) )| \, dt, \\
  B(T) &:= \int_0^T \int_{\bR^4} \frac{v^4}{|x|} \, dx \, dt,
 \end{align*}
 and note that for any $0 \leq t \leq T$ we have 
 \[
  E(v(t)) \leq E(v(0)) + A(T).  
 \]
 Since it holds that
 \begin{align*}
  \bigl\| (F+v)^3 - v^3 \bigr\|_{L^1_t L^2_x([0,T]\times\bR^4)} &\lesssim \bigl\| |F|^3 + |F| v^2 \bigr\|_{L^1_t L^2_x([0,T]\times\bR^4)} \\
  &\lesssim \|F\|_{L^3_t L^6_x([0,T]\times\bR^4)}^3 + \bigl\| |x|^{\frac{1}{2}} F \bigr\|_{L^2_t L^\infty_x([0,T]\times\bR^4)} \bigl\| |x|^{-\frac{1}{2}} v^2 \bigr\|_{L^2_t L^2_x([0,T]\times\bR^4)},
 \end{align*}
 we infer from \eqref{equ:energy_differentiated} and the Morawetz estimate~\eqref{equ:morawetz} that
 \begin{equation} \label{equ:bounds_A_B}
  \begin{aligned}
   A(T) &\lesssim \bigl( E(v(0)) + A(T) \bigr)^{\frac{1}{2}} \bigl( \|F\|_{L^3_t L^6_x([0,T]\times\bR^4)}^3 + \bigl\| |x|^{\frac{1}{2}} F \bigr\|_{L^2_t L^\infty_x([0,T]\times\bR^4)} B(T)^{\frac{1}{2}} \bigr), \\
   B(T) &\lesssim E(v(0)) + A(T) \\
   &\quad \quad + \bigl( E(v(0)) + A(T) \bigr)^{\frac{1}{2}} \bigl( \|F\|_{L^3_t L^6_x([0,T]\times\bR^4)}^3 + \bigl\| |x|^{\frac{1}{2}} F \bigr\|_{L^2_t L^\infty_x([0,T]\times\bR^4)} B(T)^{\frac{1}{2}} \bigr).
  \end{aligned}
 \end{equation}
 From \eqref{equ:bounds_A_B} and a standard continuity argument we conclude that there exists a sufficiently small absolute constant $0 < \varepsilon \ll 1$ such that if 
 \[
  \|F\|_{L^3_t L^6_x([0,T]\times\bR^4)}^3 + \bigl\| |x|^{\frac{1}{2}} F \bigr\|_{L^2_t L^\infty_x([0,T]\times\bR^4)} = \varepsilon^2,
 \]
 then it holds that
 \[
  A(T) + B(T) \lesssim E(v(0)) + \varepsilon.
 \]
 In particular, we then have 
 \[
  \sup_{0 \leq t \leq T} E(v(t)) \lesssim E(v(0)) + 1.
 \]
 By divisibility of the $L^3_t L^6_x(\bR\times\bR^4)$ norm and of the $L^2_t L^\infty_x(\bR\times\bR^4)$ norm, we may iterate this argument finitely many times to obtain that
 \[
  \sup_{t \in I_\ast} E(v(t)) \leq C \exp \Bigl( C \bigl( \|F\|_{L^3_t L^6_x(\bR\times\bR^4)}^3 + \bigl\| |x|^{\frac{1}{2}} F \bigr\|_{L^2_t L^\infty_x(\bR\times\bR^4)}^2 \bigr) \Bigr) ( E(v(0)) + 1 )
 \]
 for some absolute constant $C > 0$. We then conclude from Theorem~\ref{thm:conditional_scattering} that in fact $I_\ast = \bR$, i.e. $v(t)$ exists globally in time, and $v(t)$ scatters to free waves as $t \to \pm \infty$.
\end{proof}

\section{Proof of Theorem \ref{thm:random_scattering}: Almost sure scattering} \label{sec:almost_sure_scattering}

Finally we turn to the proof of our almost sure scattering result for the energy-critical defocusing nonlinear wave equation on~$\bR^4$. It follows from Theorem~\ref{thm:scattering} and probabilistic a priori estimates on global space-time norms of the free wave evolution of the random initial data, which we establish in this section. To this end we first recall the following large deviation estimate.

\begin{lemma}[\protect{\cite[Lemma 3.1]{BT1}}] \label{lem:large_deviation_estimate}
 Let $\{g_n\}_{n=1}^{\infty}$ be a sequence of real-valued, independent, zero-mean random variables with associated distributions $\{\mu_n\}_{n=1}^{\infty}$ on a probability space $(\Omega, {\mathcal A}, \bP)$. Assume that the distributions satisfy the property that there exists $c > 0$ such that
 \begin{equation*}
  \biggl| \int_{-\infty}^{+\infty} e^{\gamma x} d\mu_n(x) \biggr| \leq e^{c \gamma^2} \text{ for  all } \gamma \in \bR \text{ and for all } n \in \mathbb{N}.
 \end{equation*}
 Then there exists $\alpha > 0$ such that for every $\lambda > 0$ and every sequence $\{c_n\}_{n=1}^{\infty} \in \ell^2(\bN;\bC)$ of complex numbers, 
 \begin{equation*}
  \bP \Bigl( \bigl\{ \omega : \bigl| \sum_{n=1}^{\infty} c_n g_n(\omega) \bigr| > \lambda \bigr\} \Bigr) \leq 2 \exp \biggl(- \alpha \frac{\lambda^2}{\sum_n |c_n|^2} \biggr). 
 \end{equation*}
 As a consequence there exists $C > 0$ such that for every $2 \leq p < \infty$ and every $\{c_n\}_{n=1}^{\infty} \in \ell^2(\bN; \bC)$,
 \begin{equation*}
  \Bigl\| \sum_{n=1}^{\infty} c_n g_n(\omega) \Bigr\|_{L^p_\omega(\Omega)} \leq C \sqrt{p} \Bigl( \sum_{n=1}^{\infty} |c_n|^2 \Bigr)^{1/2}.
 \end{equation*}
\end{lemma}

Moreover, we present a lemma that will be used to estimate the probability of certain events. Its proof is a straightforward adaptation of the proof of Lemma~4.5 in \cite{Tz10}.

\begin{lemma} \label{lem:probability_estimate}
Let $F$ be a real-valued measurable function on a probability space $(\Omega, {\mathcal A}, \bP)$. Suppose that there exist $C_0 > 0$, $K > 0$ and $p_0 \geq 1$ such that for every $p \geq p_0$ we have
\begin{align*}
 \| F \|_{L^p_{\omega}(\Omega)} \leq \sqrt{p} \, C_0 K. 
\end{align*}
Then there exist $c > 0$ and $C_1 > 0$, depending on $C_0$ and $p_0$ but independent of $K$, such that for every $\lambda > 0$,
\begin{align*}
  \bP \bigl( \bigl\{ \omega \in \Omega : |F(\omega)| > \lambda \bigr\} \bigr) \leq C_1 e^{- c \lambda^2/K^2}.
\end{align*}
In particular, it follows that 
\[
 \bP \bigl( \bigl\{ \omega \in \Omega : |F(\omega)| < \infty \bigr\} \bigr) = 1.
\]
\end{lemma}

We may now establish a large deviation estimate on the global $L^3_t L^6_x(\bR \times \bR^4)$ norm of the free wave evolution of randomized initial data.

\begin{proposition} \label{prop:large_deviation_L3L6}
 Let $\frac{1}{3} \leq s < 1$ and let $f \in H^s_x(\bR^4)$. Denote by $f^\omega$ the randomization of $f$ as defined in~\eqref{equ:bighsrandomization}. Then there exist absolute constants $C > 0$ and $c > 0$ such that for any $\lambda > 0$ it holds that
 \begin{equation}
  \bP \Bigl( \Bigl\{ \omega \in \Omega : \bigl\| e^{\pm i t |\nabla|} P_{>4} f^\omega \bigr\|_{L^3_t L^6_x(\bR \times \bR^4)} > \lambda  \Bigr\} \Bigr) \leq C \exp \Bigl( - c \lambda^2 \bigl\| |\nabla|^s P_{>4} f \bigr\|_{L^2_x(\bR^4)}^{-2} \Bigr).
 \end{equation}
 In particular, we have for almost every $\omega \in \Omega$ that
 \begin{equation}
  \bigl\| e^{\pm i t |\nabla|} P_{>4} f^\omega \bigr\|_{L^3_t L^6_x(\bR \times \bR^4)} < \infty.
 \end{equation}
\end{proposition}
\begin{proof}
 By Minkowski's inequality, the large deviation estimate from Lemma~\ref{lem:large_deviation_estimate} and Corollary~\ref{cor:KlTa_gain_applied_Pk} we obtain for all $6 \leq p < \infty$ that
 \begin{align*}
  \bigl\| e^{\pm i t |\nabla|} P_{>4} f^\omega \bigr\|_{L^p_\omega L^3_t L^6_x} &\lesssim \sqrt{p} \biggl( \sum_{k \in \bZ^4} \bigl\| e^{\pm i t |\nabla|} P_k P_{>4} f \bigr\|_{L^3_t L^6_x}^2 \biggr)^{\frac{1}{2}} \\
  &\lesssim \sqrt{p} \biggl( \sum_{k \in \bZ^4} \bigl\| |\nabla|^{\frac{1}{3}} P_k P_{>4} f \bigr\|_{L^2_x}^2 \biggr)^{\frac{1}{2}} \\
  &\lesssim \sqrt{p} \, \bigl\| |\nabla|^s P_{>4} f \bigr\|_{L^2_x}.
 \end{align*}
 The assertion now follows from Lemma~\ref{lem:probability_estimate}.
\end{proof}

Next, we prove a large deviation estimate on a global weighted $L^2_t L^\infty_x(\bR \times \bR^4)$ norm of the free wave evolution of randomized radially symmetric initial data. This estimate is one of the main novelties of this article.

\begin{proposition} \label{prop:large_deviation_L2Linfty}
 Let $\frac{1}{2} < s < 1$ and let $f \in H^s_x(\bR^4)$ be radially symmetric. Denote by $f^\omega$ the randomization of $f$ as defined in~\eqref{equ:bighsrandomization}. Then there exist absolute constants $C > 0$ and $c > 0$ such that for any $\lambda > 0$ it holds that
 \begin{equation}
  \bP \Bigl( \Bigl\{ \omega \in \Omega : \bigl\| |x|^{\frac{1}{2}} e^{\pm i t |\nabla|} P_{>4} f^\omega \bigr\|_{L^2_t L^\infty_x(\bR \times \bR^4)} > \lambda  \Bigr\} \Bigr) \leq C \exp \Bigl( - c \lambda^2 \bigl\| |\nabla|^s P_{>4} f \bigr\|_{L^2_x(\bR^4)}^{-2} \Bigr).
 \end{equation}
 In particular, we have for almost every $\omega \in \Omega$ that
 \begin{equation}
  \bigl\| |x|^{\frac{1}{2}} e^{\pm i t |\nabla|} P_{>4} f^\omega \bigr\|_{L^2_t L^\infty_x(\bR \times \bR^4)} < \infty.
 \end{equation}
\end{proposition}
\begin{proof}
Let $\psi \in C_c^\infty(\bR^4)$ be a radial smooth bump function satisfying $\psi(x) = 1$ for $|x| \leq 1$ and $\psi(x) = 0$ for $|x| > 2$. Define for $L \in 2^{\bN}$,
\[
 \psi_L(x) := \psi(x/L) - \psi(2x/L).
\]
Then we have for any $x \in \bR^4$ that
\[
 1 = \psi(x) + \sum_{L \geq 2} \psi_{L}(x).
\]

Now let $0 < \varepsilon \leq \frac{1}{10} (s - \frac{1}{2})$ and $r > \frac{4}{\varepsilon}$. By the Sobolev embedding $W_x^{s,r}(\bR^4) \subset L^\infty_x(\bR^4)$ with $s > \frac{4}{r}$, we have for all $r \leq p < \infty$ that
\begin{align*}
 &\bigl\| |x|^{\frac{1}{2}} e^{\pm i t |\nabla|} P_{>4} f^\omega \bigr\|_{L^p_{\omega} L^2_t L^{\infty}_x} \\
 &\lesssim \bigl\| \psi(x) |x|^{\frac{1}{2}} e^{\pm i t |\nabla|} P_{>4} f^\omega \bigr\|_{L^p_{\omega} L^2_t L^{\infty}_x} + \bigl\| (1 - \psi(x)) |x|^{\frac{1}{2}} e^{\pm i t |\nabla|} P_{>4} f^\omega \bigr\|_{L^p_{\omega} L^2_t L^{\infty}_x} \\
 &\lesssim \bigl\| e^{\pm i t |\nabla|} P_{>4} f^\omega \bigr\|_{L^p_{\omega} L^2_t L^{\infty}_x} + \bigl\| (1 - \psi(x)) |x|^{\frac{1}{2}} e^{\pm i t |\nabla|} P_{>4} f^\omega \bigr\|_{L^p_{\omega} L^2_t L^{\infty}_x} \\
 &\lesssim \bigl\| \langle \nabla \rangle^{\varepsilon} e^{\pm i t |\nabla|} P_{>4} f^\omega \bigr\|_{L^p_{\omega} L^2_t L^{r}_x} + \bigl\| \langle \nabla \rangle^{\varepsilon} (1 - \psi(x)) |x|^{\frac{1}{2}} e^{\pm i t |\nabla|} P_{>4} f^\omega \bigr\|_{L^p_{\omega} L^2_t L^{r}_x} \\
 &\equiv I + II.
\end{align*}

\noindent {\bf Term I.} By Minkowski's inequality, Lemma~\ref{lem:large_deviation_estimate} and Corollary~\ref{cor:KlTa_gain_applied_Pk} we obtain for all $r \leq p < \infty$ that
\begin{align*}
  \bigl\| \langle \nabla \rangle^{\varepsilon} e^{\pm i t |\nabla|} P_{>4} f^\omega \bigr\|_{L^p_{\omega} L^2_t L^{r}_x} &\lesssim \sqrt{p} \biggl( \sum_{k \in \bZ^4} \bigl\| \langle \nabla \rangle^{\varepsilon} e^{\pm i t |\nabla|} P_k P_{>4} f \bigr\|_{L^2_t L^r_x}^2 \biggr)^{\frac{1}{2}} \\
  &\lesssim \sqrt{p} \biggl( \sum_{k \in \bZ^4} \bigl\| |\nabla|^{\frac{1}{2}} \langle \nabla \rangle^{\varepsilon} P_k P_{>4} f \bigr\|_{L^2_x}^2 \biggr)^{\frac{1}{2}} \\
  &\lesssim \sqrt{p} \, \bigl\| |\nabla|^s P_{>4} f \bigr\|_{L^2_x}.
\end{align*}

\medskip 

\noindent {\bf Term II.} We decompose this term dyadically into
\begin{align*}
 \langle \nabla \rangle^{\varepsilon} (1 - \psi(x)) |x|^{\frac{1}{2}} e^{\pm i t |\nabla|} P_{>4} f^\omega = \sum_{M > 4} \sum_N \langle \nabla \rangle^{\varepsilon} P_N \bigl( (1 - \psi(x)) |x|^{\frac{1}{2}} e^{\pm i t |\nabla|} P_M f^\omega \bigr)
\end{align*}
and we distinguish the following cases:

\medskip 

\noindent {\bf Case 1:} $N \sim M$. This is the most delicate case. By Bernstein estimates, Minkowski's inequality and Lemma~\ref{lem:large_deviation_estimate}, we have for all $r \leq p < \infty$ that
\begin{equation} \label{equ:case1_start}
 \begin{aligned}
  &\biggl\| \sum_{M > 4} \sum_{N \sim M} \langle \nabla \rangle^{\varepsilon} P_N \bigl( (1 - \psi(x)) |x|^{\frac{1}{2}} e^{\pm i t |\nabla|} P_M f^\omega \bigr) \biggr\|_{L^p_\omega L^2_t L^r_x} \\
  &\lesssim \sum_{M > 4} (1 + M^{\varepsilon}) \bigl\| |x|^{\frac{1}{2}} e^{\pm i t |\nabla|} P_M f^\omega \bigr\|_{L^p_\omega L^2_t L^r_x} \\
  &\lesssim \sqrt{p} \sum_{M > 4} M^{\varepsilon} \biggl\| \Bigl( \sum_{k \in \bZ^4} \bigl| |x|^{\frac{1}{2}} e^{\pm i t |\nabla|} P_k P_M f \bigr|^2 \Bigr)^{\frac{1}{2}} \biggr\|_{L^2_t L^r_x}.
 \end{aligned}
\end{equation}
Interpolating for any dyadic $M > 4$ the square-function estimate from Lemma~\ref{lem:radialish_sobolev}
\[
 \biggl\| \Bigl( \sum_{k \in \bZ^4} \bigl| |x|^{\frac{3}{2}} e^{\pm i t |\nabla|} P_k P_M f \bigr|^2 \Bigr)^{\frac{1}{2}} \biggr\|_{L^\infty_x} \lesssim \Bigl( \sum_{k \in \bZ^4} \bigl\| |\nabla|^{3 \varepsilon} e^{\pm i t |\nabla|} P_k P_M f \bigr\|_{L^2_x}^2 \Bigr)^{\frac{1}{2}}
\]
with the trivial bound 
\[
 \biggl\| \Bigl( \sum_{k \in \bZ^4} \bigl| e^{\pm i t |\nabla|} P_k P_M f \bigr|^2 \Bigr)^{\frac{1}{2}} \biggr\|_{L^{\frac{2r}{3}}_x} \lesssim \Bigl( \sum_{k \in \bZ^4} \bigl\| e^{\pm i t |\nabla|} P_k P_M f \bigr\|_{L^{\frac{2r}{3}}_x}^2 \Bigr)^{\frac{1}{2}},
\]
yields that
\[
 \biggl\| \Big( \sum_{k\in\bZ^4} \bigl| |x|^{\frac{1}{2}} e^{\pm i t |\nabla|} P_k P_M f \bigr|^2 \Bigr)^{\frac{1}{2}} \biggr\|_{L^r_x} \lesssim \biggl( \sum_{k\in\bZ^4} \bigl\| |\nabla|^{\varepsilon} e^{\pm i t |\nabla|} P_k P_M f \bigr\|_{L^{\frac{6r}{r+6}}_x}^2 \biggr)^{\frac{1}{2}}. 
\]
Hence, for the last line of \eqref{equ:case1_start} we obtain the bound
\begin{equation} \label{equ:case1_intermediate}
\sqrt{p} \sum_{M > 4} M^{\varepsilon} \biggl\| \Bigl( \sum_{k \in \bZ^4} \bigl| |x|^{\frac{1}{2}} e^{\pm i t |\nabla|} P_k P_M f \bigr|^2 \Bigr)^{\frac{1}{2}} \biggr\|_{L^2_t L^r_x} \lesssim  \sqrt{p} \sum_{M > 4} M^{2 \varepsilon} \biggl( \sum_{k \in \bZ^4} \bigl\| e^{\pm i t |\nabla|} P_k P_M f \bigr\|_{L^2_t L^{\frac{6 r}{r+6}}_x}^2 \biggr)^{\frac{1}{2}}.
\end{equation}
Here we encounter the obstacle that the pair $(2, \frac{6 r}{r+6})$ is not admissible. Nonetheless, a larger range of Strichartz pairs is available in the radial case. Specifically, we exploit this by interpolating $(2, \frac{6r}{r+6})$ between the endpoint admissible pair $(2, 6)$ and the radial-admissible pair $(2,4)$. Hence, by Corollary~\ref{cor:KlTa_gain_applied_Pk} and the radial Strichartz estimate from Proposition~\ref{prop:radial_strichartz_estimates}, we find that for $\theta = 1 - \frac{12}{r}$,
\begin{align*}
 \biggl( \sum_{k\in\bZ^4} \bigl\| e^{\pm i t |\nabla|} P_k P_M f \bigr\|_{L^2_t L^{\frac{6 r}{r+6}}_x}^2 \biggr)^{\frac{1}{2}}  &\lesssim \biggl( \sum_{k\in\bZ^4} \bigl\| e^{\pm i t |\nabla|} P_k P_M f \bigr\|_{L^2_t L^{6 }_x}^2 \biggr)^{\frac{\theta}{2}} \biggl( \sum_{k \in \bZ^4} \bigl\| e^{\pm i t |\nabla|} P_k P_M f \bigr\|_{L^2_t L^4_x}^2 \biggr)^{\frac{1-\theta}{2}} \\
 &\lesssim \biggl( \sum_{k\in\bZ^4} \bigl\| |\nabla|^{\frac{1}{2}} P_k P_M f \bigr\|_{L^2_x}^2 \biggr)^{\frac{\theta}{2}} \biggl( M^4 \, \bigl\| e^{\pm i t |\nabla|} P_M f \bigr\|_{L^2_t L^4_x}^2 \biggr)^{\frac{1-\theta}{2}} \\
 &\lesssim \biggl( \sum_{k\in\bZ^4} \bigl\| |\nabla|^{\frac{1}{2}} P_k P_M f \bigr\|_{L^2_x}^2 \biggr)^{\frac{\theta}{2}} \biggl( M^4 \, \bigl\| |\nabla|^{\frac{1}{2}} P_M f \bigr\|_{L^2_x}^2 \biggr)^{\frac{1-\theta}{2}} \\ 
 &\lesssim M^{\frac{1}{2} \theta} M^{2(1-\theta)} M^{\frac{1}{2}(1-\theta)} \| P_M f \|_{L^2_x} \\
 &\simeq M^{\frac{1}{2} + \frac{24}{r}} \| P_M f \|_{L^2_x} \\
 &\lesssim M^{\frac{1}{2} + 6 \varepsilon} \| P_M f \|_{L^2_x}.
\end{align*}
Thus, putting everything together, we obtain that
\begin{align*}
 \sqrt{p} \sum_{M > 4} M^{2 \varepsilon} \biggl( \sum_{k \in \bZ^4} \bigl\| e^{\pm i t |\nabla|} P_k P_M f \bigr\|_{L^2_t L^{\frac{6 r}{r+6}}_x}^2 \biggr)^{\frac{1}{2}} \lesssim  \sqrt{p} \sum_{M > 4} M^{2 \varepsilon} M^{\frac{1}{2} + 6 \varepsilon } \| P_M f \|_{L^2_x}.
\end{align*}
After an application of Cauchy-Schwarz and since $\varepsilon \leq \frac{1}{10}(s-\frac{1}{2})$, we conclude that the previous line is bounded by
\[
 \sqrt{p} \sum_{M > 4} M^{\frac{1}{2} + 8 \varepsilon} \| P_M f \|_{L^2_x} \lesssim \sqrt{p} \, \bigl\| |\nabla|^{\frac{1}{2} + 9 \varepsilon} P_{>4} f \bigr\|_{L^2_x} \lesssim \sqrt{p} \, \bigl\| |\nabla|^s P_{>4} f \bigr\|_{L^2_x}.
\]

\medskip 

\noindent {\bf Case 2:} $N \gg M$. Here we write
\begin{align*}
 \sum_{M > 4} \sum_{N \gg M} \langle \nabla \rangle^{\varepsilon} P_N \bigl( (1 - \psi(x)) |x|^{\frac{1}{2}} e^{\pm i t |\nabla|} P_M f^\omega \bigr) = \sum_{M > 4} \sum_{N \gg M} \sum_{L \geq 2} \langle \nabla \rangle^{\varepsilon} P_N \bigl( \psi_L(x) |x|^{\frac{1}{2}} e^{\pm i t |\nabla|} P_M f^\omega \bigr).
\end{align*}
Since $N \gg M$, by Fourier support considerations we must have that
\[
 P_N \bigl( \psi_L(x) |x|^{\frac{1}{2}} e^{\pm i t |\nabla|} P_M f^\omega \bigr) = P_N \Bigl( \widetilde{P}_N \bigl( \psi_L(x) |x|^{\frac{1}{2}} \bigr) e^{\pm i t |\nabla|} P_M f^\omega \Bigr).
\]
Thus, using Minkowski's inequality, Bernstein estimates, Lemma~\ref{lem:large_deviation_estimate} and Corollary~\ref{cor:KlTa_gain_applied_Pk} we obtain for all $r \leq p < \infty$ that
\begin{align*}
 &\biggl\| \sum_{M > 4} \sum_{N \gg M} \sum_{L \geq 2} \langle \nabla \rangle^{\varepsilon} P_N \Bigl( \widetilde{P}_N \bigl( \psi_L(x) |x|^{\frac{1}{2}} \bigr) e^{\pm i t |\nabla|} P_M f^\omega \Bigr) \biggr\|_{L^p_{\omega} L^2_t L^r_x} \\
 &\lesssim \sum_{M > 4} \sum_{N \gg M} \sum_{L \geq 2} N^{\varepsilon} \bigl\| \widetilde{P}_N \bigl( \psi_L(x) |x|^{\frac{1}{2}} \bigr) \bigr\|_{L^\infty_x} \bigl\| e^{\pm i t |\nabla|} P_M f^\omega \bigr\|_{L^p_\omega L^2_t L^r_x} \\
 &\lesssim \sum_{M > 4} \sum_{N \gg M} \sum_{L \geq 2} N^{-1 + \varepsilon} \bigl\| \nabla_x \widetilde{P}_N \bigl( \psi_L(x) |x|^{\frac{1}{2}} \bigr) \bigr\|_{L^\infty_x} \sqrt{p} \, \bigl\| |\nabla|^{\frac{1}{2}} P_M f \bigr\|_{L^2_x} \\
 &\lesssim \sqrt{p} \sum_{M > 4} \sum_{N \gg M} \sum_{L \geq 2} N^{-1 + \varepsilon} L^{-\frac{1}{2}} \bigl\| |\nabla|^{\frac{1}{2}} P_M f \bigr\|_{L^2_x} \\
 &\lesssim \sqrt{p} \bigl\| |\nabla|^s P_{>4} f \bigr\|_{L^2_x}.
\end{align*}

\medskip 

\noindent {\bf Case 3:} $N \ll M$. By Fourier support considerations we must have that 
\[
 P_N \bigl( \psi_L(x) |x|^{\frac{1}{2}} e^{\pm i t |\nabla|} P_M f^\omega \bigr) =  P_N \Bigl( \widetilde{P}_M \bigl( \psi_L(x) |x|^{\frac{1}{2}} \bigr) e^{\pm i t |\nabla|} P_M f^\omega \Bigr).
\]
Using Bernstein estimates we conclude that then
\begin{align*}
 &\biggl\| \sum_{M>4} \sum_{N \ll M} \sum_{L \geq 2} \langle \nabla \rangle^{\varepsilon} P_N \Bigl( \widetilde{P}_M \bigl( \psi_L(x) |x|^{\frac{1}{2}} \bigr) e^{\pm i t |\nabla|} P_M f^\omega \Bigr) \biggr\|_{L^p_{\omega} L^2_t L^r_x} \\
 &\lesssim \sum_{M>4} \sum_{N \ll M} \sum_{L \geq 2} (1 + N^{\varepsilon}) N^{\frac{2}{3}-\frac{4}{r}} \Bigl\| P_N \Bigl( \widetilde{P}_M \bigl( \psi_L(x) |x|^{\frac{1}{2}} \bigr) e^{\pm i t |\nabla|} P_M f^\omega \Bigr) \Bigr\|_{L^p_{\omega} L^2_t L^{6 }_x} \\
 &\lesssim \sum_{M>4} \sum_{N \ll M} \sum_{L \geq 2} M^{\varepsilon} N^{\frac{2}{3}-\frac{4}{r}} \bigl\| \widetilde{P}_M \bigl( \psi_L(x) |x|^{\frac{1}{2}} \bigr) \bigr\|_{L^\infty_x} \bigl\| e^{\pm i t |\nabla|} P_M f^\omega \bigr\|_{L^p_\omega L^2_t L^{6 }_x} \\
 &\lesssim \sum_{M>4} \sum_{N \ll M} \sum_{L \geq 2} M^{-1 + \varepsilon} N^{\frac{2}{3}-\frac{4}{r}} \bigl\| \nabla_x \widetilde{P}_M \bigl( \psi_L(x) |x|^{\frac{1}{2}} \bigr) \bigr\|_{L^\infty_x} \bigl\| e^{\pm i t |\nabla|} P_M f^\omega \bigr\|_{L^p_\omega L^2_t L^{6 }_x} \\
 &\lesssim \sum_{M>4} \sum_{N \ll M} \sum_{L \geq 2} M^{-1 + \varepsilon} N^{\frac{2}{3}-\frac{4}{r}} L^{-\frac{1}{2}} \bigl\| e^{\pm i t |\nabla|} P_M f^\omega \bigr\|_{L^p_\omega L^2_t L^{6 }_x}.
\end{align*}
Now using Minkowski's inequality, the large deviation estimate from Lemma~\ref{lem:large_deviation_estimate} and Corollary~\ref{cor:KlTa_gain_applied_Pk}, we find that the last line is bounded by
\begin{align*}
 \sqrt{p} \sum_{M>4} \sum_{N \ll M} \sum_{L \geq 2} M^{-1 + \varepsilon} N^{\frac{2}{3}-\frac{4}{r}} L^{-\frac{1}{2}} \bigl\| |\nabla|^{\frac{1}{2}} P_M f \bigr\|_{L^2_x} \lesssim \sqrt{p} \, \bigl\| |\nabla|^s P_{>4} f \bigr\|_{L^2_x}.
\end{align*}

\medskip 

Putting these estimates together, we have obtained that for all $r \leq p < \infty$,
\[
 \bigl\| |x|^{\frac{1}{2}} e^{\pm i t |\nabla|} P_{>4} f^\omega \bigr\|_{L^p_\omega L^2_t L^\infty_x} \lesssim \sqrt{p} \, \bigl\| |\nabla|^s P_{>4} f \bigr\|_{L^2_x}.
\]
The claim now follows from Lemma~\ref{lem:probability_estimate}.
\end{proof}

Finally, we are in a position to provide the proof of Theorem~\ref{thm:random_scattering}, which now follows readily from Theorem~\ref{thm:scattering} and the large deviation estimates established in the previous two propositions.

\begin{proof}[Proof of Theorem~\ref{thm:random_scattering}]
By Proposition~\ref{prop:large_deviation_L3L6} and Proposition~\ref{prop:large_deviation_L2Linfty} we have for almost every $\omega \in \Omega$,
\begin{equation} \label{equ:proof_of_random_scattering_large_deviation_bound1}
 \bigl\| S(t)\bigl( P_{>4} f_0^\omega, P_{>4} f_1^\omega \bigr) \bigr\|_{L^3_t L^6_x(\bR \times \bR^4)} < \infty  
\end{equation}
and
\begin{equation} \label{equ:proof_of_random_scattering_large_deviation_bound2}
 \bigl\| |x|^{\frac{1}{2}} S(t)\bigl( P_{>4} f_0^\omega, P_{>4} f_1^\omega \bigr) \bigr\|_{L^2_t L^\infty_x(\bR \times \bR^4)} < \infty.
\end{equation}
We seek a solution to the energy-critical cubic nonlinear wave equation~\eqref{equ:nlw_main_theorem} with initial data $(f_0^\omega, f_1^\omega) \in H^s_x(\bR^4) \times H^{s-1}_x(\bR^4)$ of the form
\[
 u(t) = S(t)\bigl( P_{>4} f_0^\omega, P_{>4} f_1^\omega \bigr) + v(t),
\]
where the nonlinear part $v(t)$ satisfies
\begin{equation*}
 \left\{ \begin{aligned}
          -\partial_t^2 v + \Delta v &= \bigl( S(t)\bigl( P_{>4} f_0^\omega, P_{>4} f_1^\omega \bigr) + v \bigr)^3 \text{ on } \bR \times \bR^4, \\
          (v, \partial_t v)|_{t=0} &= \bigl( P_{\leq 4} f_0^\omega, P_{\leq 4} f_1^\omega \bigr) \in \dot{H}^1_x(\bR^4) \times L^2_x(\bR^4).
         \end{aligned} \right.
\end{equation*}
Setting $F := S(t)\bigl( P_{>4} f_0^\omega, P_{>4} f_1^\omega \bigr)$, the assertion of Theorem~\ref{thm:random_scattering} follows immediately from \eqref{equ:proof_of_random_scattering_large_deviation_bound1}, \eqref{equ:proof_of_random_scattering_large_deviation_bound2} and Theorem~\ref{thm:scattering}.
\end{proof}

\newcommand{\SortNoop}[1]{}
\ifx\undefined\bysame
\newcommand{\bysame}{\leavevmode\hbox to3em{\hrulefill}\,}
\fi

\end{document}